%% file: arxiv.tex
\documentclass[ onefignum,onetabnum]{siamart220329}
\input{ex_shared}

\begin{document}

\maketitle

\begin{abstract} 
Techniques based on \( k \)-th order Hodge Laplacian operators \( L_k \) are widely used to describe the topology as well as the governing dynamics of high-order systems modeled as simplicial complexes.  In all of them, it is required to solve a number of least square problems with  \( L_k \) as coefficient matrix, for example in order to compute some portions of the spectrum or integrate the dynamical system. In this work, we introduce the notion of optimal collapsible subcomplex and we present a fast combinatorial algorithm for the computation of a sparse Cholesky-like preconditioner for $L_k$ that exploits the topological structure of the simplicial complex. The performance of the preconditioner is tested for conjugate gradient method for least square problems (CGLS) on a variety of simplicial complexes with different dimensions and edge densities. We show that, for sparse simplicial complexes, the new preconditioner reduces significantly the condition number of $L_k$ and performs better than the standard incomplete Cholesky factorization.
\end{abstract}



\section{Introduction}
\label{sec:intro}
Graph models of network systems are ubiquitous throughout the sciences. However,  a fast-growing line of work in modern network science has highlighted the importance of higher-order models that take directly into account multi-agent interactions rather than classical graph pairwise interactions. 
Simplicial complexes are a popular higher-order generalization of graphs that allows modeling non-diadic interactions by means of a sequence of nested nodal simplicies. Given the nested structure of simplicial complexes, boundary operators \( B_k \), mapping each simplex to its border, and the associated higher-order Laplacian operators \( L_k \), \cite{Lim15}, provide a fundamental description of the topology of the system. For example, one may show the convergence of \( L_k \) in the thermodynamic limit to  Hodge Laplacians on manifolds, \cite{chen2021decomposition,chen2021helmholtzian}; the elements of the kernel of the higher-order Laplacian operators correspond to \(k\)-dimensional holes in the complex (connected components, one-dimensional holes, voids, etc.); the positive part of the spectrum of \( L_k \) can be used to measure the topological stability of the system, \cite{guglielmi2023quantifying}; the corresponding eigenvectors provide information about higher-order topological features \cite{ebli2019notion,grande2023disentangling} and can be used to describe the underlying dynamics, e.g. in case of synchronization or higher-order diffusion, \cite{gambuzza2021stability, torres2020simplicial}. 
Computing (part of) the spectrum of $L_k$ with standard iterative methods heavily requires efficient solving a number of linear systems of the form \( L_k \b x = \b f \), see e.g.\ \cite{demmel1997}. Moreover, the higher-order system dynamics, for example in the case of random walks, PageRank, and opinion spreading, \cite{schaub2019random, schaub2022signal}, are primarily defined in terms of the same type of linear systems. While it was observed in \cite{cohen2014solving} that for specifically structured simplicial complexes one can efficiently solve the linear system \( L_k \b x = \b f \) directly, e.g. when the complex can be embedded in a triangulation of a unit sphere in three dimensions, attempting to extend that approach to arbitrary complex significantly increases the complexity of the direct solver, \cite{black2022computational}. 

In order to speed up the solution of linear systems with higher-order Laplacians via iterative methods, in this work we propose a sparse and computationally efficient Choelsky-like preconditioner for $L_k$ based on the notion of weak collapsible complex. 
In the case of classical Laplacian \( L_0 \), one can build a sparse stochastic Cholesky preconditioner based on the underlying graph structure, \cite{Kyng2016}. By investigating the structure of the Schur complements forming the Cholesky factorization method, we observe that the same construction transfers to the higher-order setting only for simplicial complexes with a specific weakly collapsible topological structure \cite{whitehead1939simplicial}. As a result, we propose an algorithm that builds a fast lower-triangular preconditioner based on the optimal collapsible subcomplex. While we present the problem and the main ideas in the general setting, we will focus on the case $k=1$ when designing the algorithm and testing its efficiency. 

The rest of the paper is organized as follows: we introduce formal definitions for simplicial complexes and \( k\)-Laplacian operators in \Cref{sec:simplicial_complex}; the transition from the linear system \( L_k \b x = \b f\) to the least square problems on up-Laplacians \(\Lu k\) is demonstrated in \Cref{sec:linear_systems}, \Cref{prop:LS_decomp}. \Cref{sec:cholesky} describes the Cholesky preconditioner and the relation to the weak collapsibility of the simplicial complex, \Cref{sec:collapse}. Given that, we introduce the heavy collapsible subcomplex preconditioner \algname in \Cref{sec:heavy}, based on the theoretically optimal subcomplex preconditioner introduced in \Cref{lemma:subsample_weight,thm:cond_weight}.
Finally, \Cref{sec:experiments} provides numerical results for \algname{} in terms of condition number and performance of the preconditioned conjugate gradient method, as compared to the unpreconditioned system and the standard incomplete Cholesky factorization.

\section{ Simplicial complex and Homology groups }
\label{sec:simplicial_complex}

A \textit{simplicial complex} \( \mc K \) on the vertices \( \{ v_1, v_2 \ldots v_n \} \) is a collection of simplices \( \sigma \), sets of nodes with the property that all the subsets of $\sigma$ are simplicies of $\mc K$ too.   
We refer to a simplex made out of $k$ nodes \( \sigma = [ v_{i_1}, \ldots v_{i_{k+1}} ] \) as being of order \( k \), and write \( \dim \sigma = k \); the set of all the simplices of order \( k \) in the complex \( \mc K \) is denoted by \( \V k \). Thus, \( \V 0 \) are the vertices of $\mc K$, \( \V 1 \) are edges between pairs of vertices, \( \V 2 \) triangles connecting three vertices, and so on. We let \( m_k = | \V k | \) denote the cardinality of \( \V k\).

Each set of simplices \( \V k = \left\{ \sigma_1, \dots \sigma_{ m_k } \right\} \) induces a linear space of formal sums over the simplicies \( C_k (\mc K) = \left\{  \sum_{i=1}^{ m_k } \alpha_i \sigma_i  \mid \alpha_i \in \ds R \right\} \) referred to as \emph{chain space}; in particular,  \( C_0 ( \mc K ) \) is known as the space of vertex states and \( C_1 ( \mc K ) \) as the space of edge flows. Simplices of different orders are related through the boundaries operators \( \partial_k \) mapping the simplex to its boundary; formally, \( \partial_k : C_k ( \mc K ) \mapsto C_{k-1} ( \mc K ) \) is defined through the alternating sum:
\begin{equation*}
      \partial_k [ v_1, v_2, \ldots v_k ] = \sum_{i=1}^k (-1)^{i-1} [ v_1, \ldots v_{i-1}, v_{i+1}, \ldots v_k ] 
\end{equation*} 
By fixing an ordering for \( \V k \) we can fix a canonical basis for \( C_k(\mc K)\) and represent each boundary operator as a matrix  \( B_k \in \mathrm{Mat}_{ m_{k-1} \times m_k } \) with exactly $k$ nonzero entries per each column, being either $+1$ or $-1$. For these matrices, the fundamental property of topology holds: \emph{the boundary of the boundary is zero}, {\cite[Thm.~5.7]{Lim15}}:
\begin{equation}
      \label{eq:bkbk1}
      B_k B_{k+1} = 0 
\end{equation}
The matrix representation $B_k$ of the boundary operator \( \partial_k \) requires fixing an ordering of the simplices in $\V k$ and $\V{k-1}$. As it will be particularly relevant for the purpose of this work, we emphasize that we order triangles and edges as follows: triangles in \( \V 2 \) are oriented in such a way that the first edge (in terms of the ordering of \( \V 1 \)) in each triangle is positively acted upon by $B_2$, i.e.\ the first non-zero entry in each column of \(B_2 \) is \( +1 \), \Cref{fig:orientation}. 

\begin{figure}[hbtp]
      \centering
      \includegraphics[width=1.0\columnwidth]{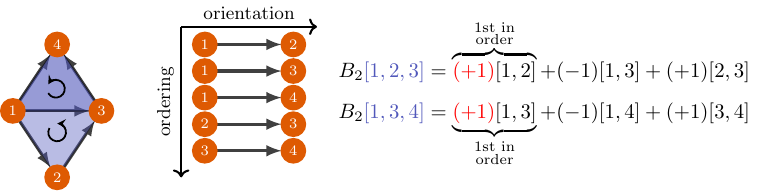}
      \caption{ Example of the simplicial complex with ordering and orientation: nodes from \( \V 0 \) in orange, triangles from \( \V 2 \) in blue. Orientation of edges and triangles is shown by arrows; the action of \( B_2 \) operator is given for both triangles.\label{fig:orientation}}
\end{figure}

The following definitions introduce the fundamental concepts of $k$-th homology group and $k$-th order Laplacian.  See \cite{Lim15} e.g.\ for more details. 

\begin{definition}[Homology group and higher-order Laplacian]
      Since \( \im B_{k+1} \subset \ker B_k \), the quotient space \( \mc H_k =  \sfrac{ \ker B_k }{ \im B_{k+1}} \), known as \( k\)-th homology group, is correctly defined and the following isomorphisms hold 
    \begin{equation*}
            \mc H_k \cong \ker B_k \cap \ker B_{k+1}^\top \cong \ker \left( B_k^\top B_k + B_{k+1} B_{k+1}^\top \right).
      \end{equation*}

      The matrix \( L_k = B_k^\top B_k + B_{k+1} B_{k+1}^\top \) is called the \(k\)-th order \emph{graph Laplacian}; the two terms \( \Ld k =  B_k^\top B_k \) and \( \Lu k = B_{k+1} B_{k+1}^\top \) are referred to as the \emph{down-Laplacian} and the \emph{up-Laplacian}, respectively.
\end{definition}

The homology group \( \mc H_k \) describes the \(k\)-th topology of the simplicial complex \( \mc K \): \( \beta_k = \dim \mc H_k = \dim \ker L_k \)  coincides exactly with the number of \(k\)-dimensional holes in the complex, known as the \emph{ \(k\)-th Betti number}. In the case \( k = 0 \), the operator \( L_0 = \Lu 0\) is exactly the classical graph Laplacian whose kernel corresponds to the \emph{connected components} of the graph, while \( \Ld 0  = 0 \). For \( k = 1 \) and \( k = 2\), the elements of \( \ker L_1 \) and \( \ker L_2\) describe the simplex 1-dimensional holes and voids respectively, and are frequently used in the analysis of trajectory data,~\cite{schaub2019random,benson2016higher}.

Although more frequently found in their purely combinatorial form, the definitions of simplicial complexes, homology groups, and higher-order Laplacians admit a generalization to the weighted case. For the sake of generality, in the rest of the work, we use the following notion of weighted boundary operators (and thus weighted simplicial complexes), as considered in e.g.~\cite{guglielmi2023quantifying}.

\begin{definition}[Weighted and normalised boundary matrices]
       For \emph{weight functions} \( w_k : \V k \mapsto \ds R_+ \cup \{ 0 \} \), define the diagonal weight matrix \( W_k \in \mathrm{Mat}_{ m_k \times m_k } \) as  \( (W_k)_{ii} = \sqrt{w_k(\sigma_i)}\). Then the  weighting scheme for the boundary operators upholding the Hodge algebras~\eqref{eq:bkbk1} is given by:
      \begin{equation}
            \label{eq:weighting}
            B_k \mapsto W_{k-1}^{-1} B_k W_k
      \end{equation}
\end{definition}

Note that, with the weighting scheme \Cref{eq:weighting}, the dimensionality of the homology group is preserved, \( \dim \ker L_k = \dim \ker \widehat L_k \),~\cite{guglielmi2023quantifying} as well as the fundamental property of topology \Cref{eq:bkbk1}. For the sake of simplicity, in the following we will omit writing the weighting matrices explicitly and will write $B_k$ to denote a generic possibly weighted boundary operator $B_k=W_{k-1}^{-1} B_k W_k$.

\section{Linear systems}\label{sec:linear_systems}
Note that \(k\)-th order Laplacians \( L_k \), \( k > 0 \), are matrices with a large number of zero entries since \( m_k \ll m_{k-1}^2\) asymptotically  (e.g. the number of triangles \( m_2 \) is always bounded by \(m_1^{3/2}\)). In particular, in analogy with the standard $k=0$ case, we say that the simplicial complex \( \mc K \) and the corresponding Laplacians \( L_k \) as sparse if  \( m_{k} = \mc O( m_{k-1}\log m_{k-1}) \),  that is the number of simplices of higher order is comparable up to a constant times the number of simplices of lower order times its logarithm. 
This is, for example, the case for structured simplicies such as trees for the classical Laplacian, or triangulations for \( L_1 \); similarly, one can show the existence of a sparser simplicial complex approximating the original one with the \( m_{k} = \mc O( m_{k-1}\log m_{k-1}) \) sparsity pattern, \cite{osting2017spectral, spielman2008graph}.

The spectrum of the \(k\)-th order Laplacian defines the \(k\)-th topology of the complex \( \mc K \). In most applications, one is actually interested in just a few nonzero eigenvalues, possibly in the leftmost or the rightmost part of the spectrum \cite{grande2023topological,guglielmi2023quantifying,ribandogros2023combinatorial}. 
For large and sparse 
simplicial complexes, the methods of choice for the computation of these eigenspaces are iterative eigensolvers that require solving a (possibly large) number of linear systems (least-squares problems) of the form \( L_k \b x = \b f \), for suitable $\b x,\b f \bot \ker L_k$.

The following elementary result shows that solving \( L_k \b x = \b f \) can be brought down to the solution of \( \Lu k \b x = \b f \). We provide the proof for completeness. 

\begin{theorem}[Joint \(k\)-Laplacian solver]
\label{prop:LS_decomp}
     The linear system \( L_k \b x = \b f \) can be reduced to a sequence of consecutive least square problems for isolated up-Laplacians. Precisely, \( \b x \) is a solution of system \eqref{eq:exact_system},
     \begin{equation}
            \label{eq:exact_system}
            L_k \b x = \b f \qquad \text{s.t.} \qquad 
              \b x, \b f \perp \ker L_k    
     \end{equation}
     if and only if it can be written as \( \b x = B_k^\top \b u + \b x_2 \), where:
     \begin{equation*}
     \!\wh{\b u} = \underset{\b z}{\argmin} \left\| \Lu {k-1} \b z - B_k \b f_1 \right\|, \quad \!\! \b u = \underset{\b z}{\argmin} \left\| \Lu {k-1} \b z - \wh{\b u} \right\|, \quad \!\!  \b x_2 = \underset{\b y}{\argmin} \left\| \Lu k \b y - \b f_2 \right\|
     \end{equation*}
     and \( \b f =  \b f_1 + \b f_2 \) with \( \b f_1 = B_k^\top \b z_1 \), \( \b z_1 = \underset{\b z}{\argmin} \left\|  \Lu {k-1} \b z - B_k \b f \right\| \).
\end{theorem}
\begin{proof}
For the \(k\)-th order Hodge Laplacian \( L_k  \), it holds (see \cite{Lim15} e.g.):
      \begin{equation}
            \label{lem:hodge_decomp}
            \ds R^{m_{k-1}} = \lefteqn{\overbrace{\phantom{\im B_k^\top \oplus  \ker L_k}}^{\ker B_{k+1}^\top}} \im B_k^\top \oplus
            \underbrace{\ker L_k \oplus  \im B_{k+1}}_{\ker B_k}\,.
      \end{equation}
	Thus, \( \b x = \b x_1 + \b x_2 = B_k^\top \b u + B_{k+1} \b v \) for some \( \b u\), \( \b v\). Then, the system \( \Lu k \b x = \b f \) is equivalent to:
	\begin{equation*}
		\begin{cases}
			\Ld k B_k^\top \b u = B_k^\top B_k B_k^\top \b u = \b f_1 \\
			\Lu k B_{k+1} \b v = B_{k+1} B_{k+1}^\top B_{k+1} \b v = \b f_2 \\
			\b f_1 + \b f_2 = \b f
		\end{cases}
	\end{equation*}
	where \( \b f_1 \) and \( \b f_2 \) are the Hodge decomposition of the right hand side \( \b f \), with \( \b f_1 \in \im B_k^\top \) and \( \b f_2 \in \im B_{k+1} \). Then \( \b f = B_k^\top \b z_1 + B_{k+1} \b z_2 \) and, after multiplication by \( B_k \), \( B_k B_k^\top \b z_1 = \Lu {k-1} \b z_1 = B_k \b f  \iff \min_{\b z_1} \left\|  \Lu {k-1} \b z_1 - B_k \b f \right\|\) and \( \b f_2 = \b f -  B_k^\top \b z_1\).
	
	Finally, we note that  equation \( B_k^\top B_k B_k^\top \b u = \b f_1 \iff ( \Lu {k-1})^2 \b u = B_k \b f_1\)  can be solved by two consecutive LS-problems:
	\begin{equation*}
			\min_{ \wh{ \b u } } \left\| \Lu {k-1} \wh{ \b u } - B_k \b f_1 \right\|, \qquad 
			\min_{ { \b u } } \left\| \Lu {k-1} { \b u } - \wh{ \b u} \right\|
	\end{equation*}
	which corresponds to the solution of the down part of the original system. 
\end{proof}

 \Cref{prop:LS_decomp} reduces the problem of solving a linear system (least-squares problem) for the entire Laplacian matrix  \( L_k\) to the problem of solving a system for the up-Laplacian \( \Lu k \). Thus, in the next sections, we propose an efficient combinatorial Cholesky-like preconditioner for \(\Lu k\).

\section{Efficient solver for the up-Laplacian}\label{sec:cholesky}

When the up-Laplacian least-squ\-a\-re problem \( \min_{\b x} \| \Lu k \b x - \b f  \| \) is solved with standard fast iterative methods such as CGLS, \cite{bjorck1998stability,hestenes1952methods},  the convergence rate and overall behavior of the solver highly depend on the condition number of \( \Lu k \), defined as 
\[
\kappa ( \Lu k ) = \frac{ \sigma_{\max} (\Lu k) }{ \sigma_{\min} (\Lu k )}\, ,
\]
where \( \sigma_{\max}\) and \( \sigma_{\min}\) are the largest and smallest positive singular values of \( \Lu k\).

In order to reduce \( \kappa ( \Lu k ) \), we move from $\min_{\b x}\|\Lu k \b x -\b f\|$ to the symmetrically preconditioned system 
\begin{equation*}
     \min_{\b x}\|\left( C^+ \Lu k  ( C^{+})^\top \right) \left( C^\top \b x  \right) - C^+ \b f\|
\end{equation*}
 where $C^+$ denotes the Moore-Penrose pseudoinverse of $C$, and  $C$ is chosen so that 
(a) the transition between unconditioned and preconditioned systems is bijective,
(b) the matrix \( C^+ \Lu k (C^{+})^\top \) is better conditioned than the initial $\Lu k$, and (c) the pseudo-inverse of the preconditioner \( C \) can be efficiently implemented.
In particular, a cheap {\tt matvec} of the pseudo-inverse \( C^+ \b f \) is guaranteed in the case of \( C \) being a lower-triangular matrix. This is the case, for example, for the \emph{Cholesky} and the \emph{incomplete Cholesky} preconditioners \cite{golub2013matrix,higham1990analysis,manteuffel1980incomplete}.

Following this line of approach, a variety of techniques have been developed for preconditioning the Laplacian matrix \( \Lu 0 = L_0 \). In particular, one of the most prominent approaches is based on the celebrated spectral sparsification result by Spielman and Srivastava,~\cite{spielman2008graph}.  This result has been recently extended to the higher-order Laplacians,~\cite{osting2017spectral}. However, while the step from sparsification to preconditioner is well understood in the \( k = 0 \) case, when \( k > 0 \), efficiently computing an effective preconditioner from a sparsified simplicial complex is not trivial.

In the next subsection, we review how a sparse approximate Cholesky preconditioner is computed from the spectral sparsifier in the graph case. We then highlight the difficulties for \( k > 0 \) and will devote the rest of the paper to design of an efficient approach for \(k=1\).

\subsection{Schur Complements and Cholesky preconditioner for \( k = 0 \)}

We start with a useful observation about the structure of the Gaussian elimination process to form the Cholesky factor. Let \( A \) be a real symmetric positive definite matrix. Then its \emph{Schur complements} \( S_i \) obtained in the process of Cholesky factorization via Gaussian elimination can be defined recursively as follows: Let  \( \delta_i \) be the \(i\)-th canonical vector, \( (\delta_i)_i = 1\) and \( ( \delta_i )_j = 0\), \(i \ne j \), then set $S_0=A$ and for $i=1,2,3,\dots$ 
\begin{equation}
\begin{aligned}
      \label{eq:gauss}
      & S_i = S_{i-1} - \frac{ 1 }{ \delta_i^\top S_{i-1} \delta_i } S_{i-1} \delta_i \delta_i^\top S_{i-1}^\top \\
      & \b c_i = \frac{1}{\sqrt{ \delta_i^\top S_{i-1} \delta_i }} S_{i-1} \delta_i
\end{aligned}
\end{equation}
With these definitions, the Cholesky factor $C$ such that $A=CC^\top$ is formed by the columns $\b c_i$, namely 
            \[ C = [ \b c_1 \, \b c_2 \, \ldots \b c_m ].\] 
Using induction, one can directly verify that \( C \) defined above is indeed lower triangular. Moreover, since \( S_i - S_{i-1} = \b c_i \b c_i^\top\) and \( S_m  = 0 \), we have \( A = S_0 = S_0 - S_m = \sum (S_i - S_{i-1} ) = \sum \b c_i \b c_i^\top = C C^\top \).      

The case of classical Laplacian \( L_0 \) is uniquely suited for Schur complements: indeed, in the case \( k = 0 \), each Schur complement  \( S_i \) has the form \( S_i = H_i + K_i \) where both \( H_i \) and \(K_i\) are graph Laplacians on the same set of vertices.
By this key observation, one can construct a stochastic Cholesky preconditioner for \( L_0 \): each \(i\)th step of Gauss elimination substitutes the dense Laplacian \( K_i \) with a sparser Laplacian on a subgraph around the vertex \( i \), \cite{Kyng2016}, to produce fast and sparse approximate Cholesky multiplier.
This property unfortunately holds only for \( L_0 \) and does not carry over to \( L_k \) for \( k > 0 \), as we demonstrate in the next subsection.

\subsection{The structure of the Schur complements $S_i$ for $k=1$}

One of the keys to the decomposition $S_i=H_i+K_i$ for the Schur complements of  \( L_0 \) is the decomposition \( L_0 = \sum_e L_0(e)\), where the sum goes over all the edges of the graph and \( L_0(e) \) is the rank-1 graph Laplacian corresponding to the edge \( e \), \( L_0(e) = w_1(e) \b e \b e^\top\) with \( \b e \) being the corresponding incidence column from the operator \( W_0^{-1} B_1 \). 
One can directly show an analogous decomposition into rank-1 terms for any up-Laplacian:
\begin{lemma}[Rank-1 decomposition of \( \Lu k \)]\label{lem:decomposition_Luk}
      For a simplicial complex \( \mc K \), the following decomposition into rank-1 up-Laplacian terms holds for any \( \Lu k \):
      \begin{equation*}
            \Lu k = \sum_{ t \in \V k } \Lu k ( t ) = \sum_{ t \in \V k } w_k(t) \b e_t \b e_t^\top
      \end{equation*}
      where, for each $t\in \V k$,  \( \Lu k(t) = w_k(t) \b e_t \b e_t^\top\) is the rank-1 up-Laplacian for each triangle \( t \) such that (a) \( w_k(t) = [W_k^2]_{tt} \) is the weight of the simplex $t$,  and (b)  \( \b e_t \) is the vector obtained as the action of the boundary operator $W_k^{-1} B_k$ on \( t\), i.e.~the \(t\)-th column of the matrix \( W_k^{-1} B_k \).
\end{lemma}

Using \Cref{lem:decomposition_Luk}, we obtain the following characterization for the first Schur complement $S_1$ for $\Lu 1$:
\begin{lemma}[First Schur complement \( S_1 \) for \( \Lu 1 \)]
      \label{lemma:schur1}
      For the up-Laplacian \( \Lu 1\), the following derivation of the first Schur complement holds:
      \begin{equation*}
            S_1 = \sum_{t \mid 1 \notin t } w_2(t) \b e_t \b e_t^\top + \frac{ 1 }{ 2 \Omega_{ 1 } } \sum_{\substack{ t_1 \mid 1 \in t_1 \\  t_2 \mid 1 \in t_2 }} w_2(t_1) w_2(t_2) \big[ \b e_{t_1} - \b e_{t_2} \big] \big[ \b e_{t_1} - \b e_{t_2} \big]^\top
      \end{equation*}
      where \( \Omega_{  1 }\) is the total weight of all triangles adjacent to the edge \(1\), \( \Omega_{ 1 } = \sum_{ t \mid 1 \in t } w_2(t) \). 
\end{lemma}
\begin{proof}
Following \Cref{eq:gauss}, one needs to compute the constant \( \delta_1^\top S_0 \delta_1 \) and the rank-1 matrix \( S_0 \delta_1 \delta_1^\top  S_0 \) where \( S_0 = \Lu 1\). Note that \( \b e_t^\top \delta_1 = \frac{1}{\sqrt{w_1(1)}} \ds 1_{1 \in t}\), the indicator of the triangle \( t \) having the edge \( 1 \), since \(1\) is necessarily the first (in the chosen order) edge in the triangle; hence, \( \sum_{  t } w_2(t) \b e_t^\top \delta_1 =  \frac{1}{\sqrt{w_1(1)}}\sum_{ t \mid 1 \in t  } w_2(t)  \): 
      \begin{equation*}
            \begin{aligned}
                  \delta_1^\top S_0 \delta_1 & = \delta_1^\top \left(  \sum_{t } w_2(t) \b e_t \b e_t^T \right) \delta_1 = \sum_{ t } w_2(t) \left( \delta_1^\top \b e_t \right)^2 =  \frac{1}{w_1(1)}\sum_{ t \mid 1 \in t } w_2(t) = \frac{ \Omega_{ \{ 1 \} \mid \varnothing}}{w_1(1)}
            \end{aligned}
      \end{equation*}
      \begin{equation*}
            \begin{aligned}
                  S_0 \delta_1 \delta_1^\top S_0 & = \sum_{t_1, t_2}  w_2(t_1) w_2(t_2) \b e_{t_1} \b e_{t_1}^\top  \delta_1 \delta_1^\top   \b e_{t_2} \b e_{t_2}^\top  = \frac{1}{w_1(1)} \sum_{\substack{ t_1 \mid 1 \in t_1 \\  t_2 \mid 1 \in t_2 }}  w_2(t_1) w_2(t_2) \b e_{t_1} \b e_{t_2}^\top \, .
            \end{aligned}
      \end{equation*}
      By symmetry,
      \[  \sum_{\substack{ t_1 \mid 1 \in t_1 \\  t_2 \mid 1 \in t_2 }}  w(t_1) w(t_2) \b e_{t_1} \b e_{t_2}^\top = \frac{1}{2}  \sum_{\substack{ t_1 \mid 1 \in t_1 \\  t_2 \mid 1 \in t_2 }}  w(t_1) w(t_2) \left( \b e_{t_1} \b e_{t_2}^\top + \b e_{t_2} \b e_{t_1}^\top \right), \] 
      one can note by a straightforward arithmetic transformation that
      \begin{equation*}
            \b e_{t_1} \b e_{t_2}^\top + \b e_{t_2} \b e_{t_1}^\top = - \big[ \b e_{t_1} - \b e_{t_2} ] \big[ \b e_{t_1} - \b e_{t_2} ]^\top + \b e_{t_1} \b e_{t_1}^\top + \b e_{t_2} \b e_{t_2}^\top
      \end{equation*}
      Finally, \( \sum_{\substack{ t_1 \mid 1 \in t_1 \\  t_2 \mid 1 \in t_2 }}  w(t_1) w(t_2) \b e_{t_1} \b e_{t_1}^\top = \Omega_{ 1 } \sum_{ t \mid 1 \in t  } w(t) \b e_t \b e_t^\top \).
\end{proof}

The structure of the first Schur complement \( S_1 \) is reminiscent of the classical graph Laplacian case as $S_1$ is formed by the sum of two terms: 
\[
H_1 = \sum_{t \mid 1 \notin t } w_2(t) \b e_t \b e_t^\top, \qquad K_{1}= \frac{ 1 }{ 2 \Omega_{ 1 } } \sum_{\substack{ t_1 \mid 1 \in t_1 \\  t_2 \mid 1 \in t_2 }} w_2(t_1) w_2(t_2) \big[ \b e_{t_1} - \b e_{t_2} \big] \big[ \b e_{t_1} - \b e_{t_2} \big]^\top.
\]
As in the \( k =0 \) case,  the term $H_1$, which corresponds to the portion of the original \( \Lu 1 \) not adjacent to the edge being eliminated (edge \( 1 \)), is an up-Laplacian. The remaining term of the original \( \Lu 1 \) which consists of up-Laplacians adjacent to edge \(1\), is transformed here into the matrix \( K_{ 1 } \), which we refer to as the \emph{cyclic term}. Unfortunately, unlike the $0$-Laplacian case, it is easy to realize that the cyclic term \( K_{ 1 } \) is generally not an up-Laplacian. We show this with a simple illustrative example below.

\begin{example}
      Assume the simplicial complex formed by two triangles, \( \V 0 = \{ 1, 2, 3, 4\} \), \( \V 1 = \{ 12, 13, 14, 23, 24 \} \) and \( \V 2 = \{ 123, 124 \} \), adjacent by the first edge \( 12 \) with \( W_1 = W_2 = I \). Then, vectors \( \b e_{t_1} = \begin{mt} 1 & -1 & 0 & 1 & 0 \end{mt}^\top \), \( \b e_{t_2} = \begin{mt} 1 & 0 & -1 & 0 & 1  \end{mt}^\top \) and \( K_{ 1 } = \begin{mt} 0 & -1 & 1 & 1 & -1  \end{mt} \cdot \begin{mt} 0 & -1 & 1 & 1 & -1  \end{mt}^\top \). Note that \(K_{ 1 } \) is denser than \( \Lu 1 \) and has lost the structural balance of \( \Lu 1\): \(\diag K_{ 1 } = K_{ 1 } \b 1 = \b 0 \) where \( \diag \Lu 1 = \Lu 1 \b 1 = \left( \Omega_{ i } \right)_{i=1}^{m_2 }  \).
\end{example}

Further computation of Schur complements \( S_2, S_3, S_4 \dots \) unavoidably introduces a fastly growing number of similar cyclic terms.
Thus, to build an efficient preconditioner for \( \Lu 1 \), one needs to deal with the cyclic terms arising on each step \( S_i\). 

To this end, note that \( K_{ 1 } \) is zero if and only if there is a unique triangle \( t \in \V 2 \) adjacent to the edge \( 1 \). Maintaining this property for all Schur complements is equivalent to the concept of weak collapsibility which we introduce next.

\section{Collapsibility of a Simplicial Complex}
\label{sec:collapse}

In this section we borrow the terminology from \cite{whitehead1939simplicial} to introduce the concept of collapsibility and then weak collapsibility of a simplicial complex. 
 The simplex \(\sigma \in \mc K \) is \emph{free} if it is a face of exactly one simplex \(\tau = \tau(\sigma) \in \mc K \) of higher order (maximal face).  The \emph{collapse} \( \mc K \backslash \{ \sigma \} \) of \( \mc K \) at a free simplex \( \sigma \) is the operation of reducing $\mc K$ to $\mc K'$, where \( \mc K' = \mc K - \sigma - \tau \); namely, this is the operation of removing a simplex \( \tau \) having an accessible (not included in another simplex) face \(\sigma\).
 A sequence of collapses done at the simplicies  \( \Sigma = \{ \sigma_1, \sigma_2, \ldots \} \) is called a \emph{collapsing sequence}; formally:
  \begin{definition}[Collapsing sequence]
       Let \( \mc K \) be a simplicial complex. We say that \( \Sigma = \{ \sigma_1, \sigma_2, \ldots \} \) is a \emph{collapsing sequence} for $\mc K$ if \( \sigma_1 \) is free in \( \mc K =\mc K^{(1)}\) and each \( \sigma_i \), \( i > 1 \), is free at 
       \( \mc K^{(i)} = \mc K^{(i-1)} \backslash \{ \sigma_i \} \). The resulting complex $\mc L$ obtained collapsing $\mc K$ at $\Sigma$ is denoted by \( \mc L = \mc K \backslash \Sigma \).
 \end{definition}
  Note that, by  definition, every collapsing sequence \( \Sigma \) has a corresponding sequence \( \ds T = \{ \tau(\sigma_1), \tau(\sigma_2), \ldots \} \) of maximal faces being collapsed at every step. The notion of \emph{collapsible simplicial complex} is defined in \cite{whitehead1939simplicial} as follows
 \begin{definition}[Collapsible simplicial complex]
      The simplicial complex \( \mc K \) is collapsible if there exists a collapsing sequence \( \Sigma \) such that \( \mc K \) collapses to a single vertex, i.e.\ \( \mc K \backslash \Sigma = \{ v \} \) for some $v\in \V0$.
\end{definition}

While least square problems with collapsible simplicial complexes can be solved directly in an efficient way, \cite{cohen2014solving}, collapsibility is a strong requirement for a simplicial complex. In fact, determining whether the complex is collapsible is in general \emph{NP-complete}, \cite{tancer2016recognition}, even though it an be almost linear for a specific set of families of \( \mc K \), \cite{cohen2014solving}. Moreover, simplicial complexes are rarely collapsible, as we discuss in the following.

Next, we recall the concept of a \(d\)-Core, \cite{tancer2008dcollapse}:
\begin{definition}[\(d\)-Core]
      A \(d\)-Core is a subcomplex of \( \mc K \) such that every simplex of dimension \( d - 1\) belongs to at least \( 2 \) \(d\)-simplices.
\end{definition}
So, for example,  a \(2\)-Core of a $2$-skeleton $\mc K$, is a subcomplex of the original complex \( \mc K \) such that every edge from \( \V 1 \) belongs to at least \(2\) triangles from \( \V 2 \). Finally, we say that \( \mc K \) is \(d\)-collapsible if it can be collapsed only by collapses at simplices \( \sigma_i \) of order smaller than \(d\), i.e. \( \dim \sigma_i \le d-1 \). Then we have the following criterium: 

\begin{lemma}[\cite{lofano2019worst}]
      \( \mc K \) is \(d\)-collapsible if and only if it does not contain any \( d\)-core.
\end{lemma}

The \(d\)-Core is the generalization of the cycle for the case of \(1\)-collapsibility of a classical graph, and finding a \(d\)-Core inside a complex \( \mc K \) is neither trivial nor computationally cheap. Note that a \(d\)-Core is in general dense, due to its definition, and does not have a prescribed structure. We illustrate simple exemplary cores in the case of $d=2$ in \Cref{fig:2-core}, hinting at the combinatorial many possible configurations for a general \(d\)-Core, for \( d \ge 2 \). 
\begin{figure}[htbp]
      \centering
      \includegraphics[width=0.75\columnwidth]{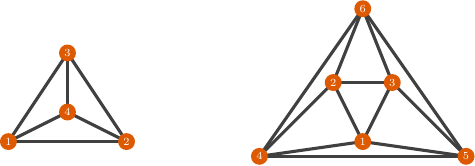}
      \caption{ \(2\)-Core, examples: all \(3\)-cliques in graphs are included in corresponding \( \V 2 \). \label{fig:2-core} }
\end{figure}

Due to the appearance of \( d\)-Core being a local phenomenon, an arbitrary simplicial complex \(\mc K\) tends to contain \(2\)-Cores as long as \( \mc K \) is denser than the trivially collapsible case.

While collapsibility is a strong requirement, in the next section, we show that a weaker condition is enough to efficiently design a preconditioner for any ``sparse enough'' simplicial complex in the \(2\)-skeleton case. Whilst the extension to \(k > 2\) is not hard, a careful reconsideration of necessary modifications in the proof of \Cref{thm:poly} may be required.

\subsection{Weak collapsibility}

Let the complex \( \mc K \) be restricted up to its \(2\)-skeleton, \( \mc K = \V 0 \cup \V 1 \cup \V 2 \), and assume \( \mc K \) is collapsible. Then a  collapsing sequence \( \Sigma \) necessarily involves collapses at simplices \( \sigma_i \) of different orders: at edges (eliminating \emph{edges} and \emph{triangles}) and at vertices (eliminating \emph{vertices} and \emph{edges}). One can show that for a given collapsing sequence \( \Sigma \) there is a reordering \( \tilde \Sigma \) such that \( \dim \tilde{\sigma_i} \) in the reordered sequence are non-increasing, {\cite[Lemma 2.5]{cohen2014solving}}. Namely, if such a complex is collapsible, then there is a collapsible sequence \( \Sigma = \{ \Sigma_1, \Sigma_0 \} \) where \( \Sigma_1 \) contains all the collapses at edges first and \( \Sigma_0 \) is composed of collapses at vertices. Note that the partial collapse \( \mc K \backslash \Sigma_1 = \mc L \) eliminates all the triangles in the complex, \( \mc V_2 (\mc L) = \varnothing \); otherwise, the whole sequence \( \Sigma \) is not collapsing \( \mc K \) to a single vertex. Since \( \mc V_2 (\mc L ) = \varnothing \), the associated up-Laplacian \( \Lu 1 ( \mc L ) = 0 \).

\begin{definition}[Weakly collapsible complex]
      A simplicial complex \( \mc K \) restricted to its \(2\)-skeleton is called \emph{weakly collapsible}, if there exists a collapsing sequence \( \Sigma_1 \) such that the simplicial complex \( \mc L = \mc K \backslash \Sigma_1 \) has no simplices of order \(2\), i.e. \( \mc V_2(\mc L) = \varnothing \) and \( \Lu 1 (\mc L ) = 0 \).
\end{definition}

      Note that while a collapsible complex is necessarily weakly collapsible, the opposite does not hold. Consider the following example in \Cref{fig:weak_example}: the initial complex is weakly collapsible either by a collapse at \( [3, 4] \) or at \( [2, 4] \). After this, the only available collapse is at the vertex \([4]\) leaving the uncollapsible \(3\)-vertex structure.
      \begin{figure}[htbp]
            \centering
            \includegraphics[width=1.0\columnwidth]{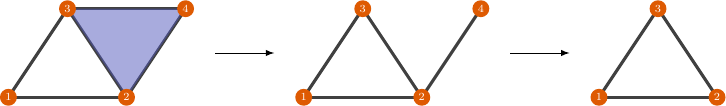}
            \caption{Example of weakly collapsible but not collapsible simplicial complex \label{fig:weak_example}}
      \end{figure}

For a given simplicial complex \( \mc K \), one can use the following greedy algorithm in order to find a collapsing sequence \(\Sigma \) and test for collapsibility: at each iteration perform any of possible collapses; in the absence of free edges, the complex should be considered not collapsible. This greedy procedure is illustrated in \Cref{algo:greedy}. 

Specifically, let \( \Delta_\sigma \) be a set of triangles of \( \mc K \) containing the edge \( \sigma \), \( \Delta_\sigma = \{ t \mid  t \in \V 2  \text{ and } \sigma \in t \} \). Then the edge \( \sigma \) is free and has only one adjacent triangle \( \tau = \tau(\sigma)\) iff \( | \Delta_\sigma | = 1 \); let \( F = \{ \sigma \mid | \Delta_e | = 1  \} \) is a set of all free edges. Note that \( | \Delta_e | \le m_0 -2 = \mc O ( m_0  ) \).

\begin{algorithm}[h]
      \caption{ \texttt{GREEDY\_COLLAPSE}(\(\mc K\)):  greedy algorithm for the weak collapsibility
      \label{algo:greedy}}
      \begin{algorithmic}[1]
            \Require initial set of free edges \( F \), adjacency sets \(  \{ \Delta_{ \sigma_i } \}_{i=1}^{ m_1 } \)
             \State \( \Sigma = [ \; ] , \; \ds T = [ \; ] \) \Comment{ initialize the collapsing sequence}
             \While{ \( F \ne \vn \) \textbf{ and } \( \V 2 \ne \vn \) }
                  \State \( \sigma \gets \texttt{pop}( F ) \), \( \; \tau \gets \tau(\sigma)  \) \Comment{ pick a free edge \( \sigma \) }
                  \State \( \mc K \gets \mc K \backslash \{ \sigma \} \), \( \; \Sigma \gets [ \, \Sigma \;\; \sigma \, ] , \; \ds T \gets [ \ds T \; \tau ] \) \Comment{ \small \( \tau \) is a triangle being collapsed; \( \tau = [ \sigma, \sigma_1, \sigma_2 ] \) }
                  \State \( \Delta_{\sigma_1} \gets \Delta_{\sigma_1} \backslash \tau  \), \( \; \Delta_{\sigma_2} \gets \Delta_{\sigma_2} \backslash \tau  \) \Comment{ remove \( \tau \) from adjacency lists }
                  \State \( F \gets F \cup \{ \sigma_i \, |  \, i = 1, 2 \text{ and } | \Delta_{\sigma_i} | = 1 \}  \) \Comment{ update \( F \) if any of \( \sigma_1 \) or \( \sigma_2 \) has become free }
             \EndWhile
             \State \Return \( \mc K, \, \Sigma, \, \ds T \)
      \end{algorithmic}
\end{algorithm}

\Cref{algo:greedy} recursively picks up a free edge \( \sigma \) from the set \( F \), performs a collapse \( \mc K \gets \mc K \backslash \{ \sigma \} \) and updates \( F \) for the collapsed subcomplex. Then, the greedy approach may fail to find the collapsing sequence only if it gets stuck on the collapsible complex, so the order of collapses matters. We demonstrate the validity of the greedy \Cref{algo:greedy} in the following theorem:

\begin{theorem}\label{thm:poly}
      \Cref{algo:greedy} finds a weakly collapsible sequence of \( 2\)-skeleton simplicies in polynomial time. Additionally, it finds a collapsing sequence if and only if the simplicial complex is collapsible. 
\end{theorem}

\begin{proof} 
      Clearly, \Cref{algo:greedy} runs polynomially with respect to the number of simplexes in \(\mc K\), so its consistency automatically yields polynomiality.

      The failure of the greedy algorithm would indicate the existence of a weakly collapsible complex \( \mc K \) such that the greedy algorithm gets stuck at a \( 2 \)-Core, which is avoidable for another possible order of collapses. Among all such complexes, let \( \mc K \) be any minimal one with respect to the number of triangles \( m_2 \). Then there exiss a free edge \( \sigma \in \V 1 \) such that \( \mc K \backslash \{ \sigma \} \) is \emph{collapsible} and another free edge \( \sigma' \in \V 2 \) such that \( \mc K \backslash \{ \sigma' \} \) is \emph{not collapsible}.

      Note that if \( \mc K \) is minimal then any pair of free edges \( \sigma_1 \) and \( \sigma_2 \) belong to the same triangle: \( \tau(\sigma_1) = \tau(\sigma_2) \).

      Indeed, for any \( \tau(\sigma_1) \ne \tau(\sigma_2) \),  \( \mc K \backslash \{ \sigma_1, \sigma_2 \} = \mc K \backslash \{ \sigma_2, \sigma_1 \} \). Let \( \tau(\sigma_1) \ne \tau(\sigma_2) \) for at least one pair of \( \sigma_1 \) and \( \sigma_2 \); in our assumption, either (1) both \( \mc K \backslash \{ \sigma_1 \} \) and \( \mc K \backslash \{ \sigma_2 \} \), (2) only \( \mc K \backslash \{ \sigma_1 \}  \) or (3) none are collapsible. 
      
      In the first case, either \( \mc K \backslash \{ \sigma_1 \} \) or \( \mc K \backslash \{ \sigma_2 \} \) is a smaller example of the complex satisfying the assumption above (since a bad edge \( \sigma' \) can not be either \( \sigma_1\) or \( \sigma_2\) and belongs to collapsed complexes), hence, violating the minimality.
      
      If only \( \mc K \backslash \{ \sigma_1 \} \) is collapsible, then \( \mc K \backslash \{ \sigma_2, \sigma_1 \}  \) is not collapsible (since \( \mc K \backslash \{ \sigma_2\} \) is not collapsible and \( \sigma_1 \) is free in \( \mc K \) and in \( \mc K \backslash \{ \sigma_2\} \) assuming \( \tau(\sigma_1) \ne \tau(\sigma_2)\)); hence, \( \mc K \backslash \{ \sigma_1, \sigma_2 \} \) is not collapsible (since \( \mc K \backslash \{ \sigma_1, \sigma_2 \}  = \mc K \backslash \{ \sigma_2, \sigma_1 \}  \) as we stated above), so \( \mc K \backslash \{ \sigma_1 \} \) is a smaller example of a complex satisfying the assumption. 
      
      Finally, if both \( \mc K \backslash \{ \sigma_1 \} \) and \( \mc K \backslash \{ \sigma_2 \} \) are collapsible, then for known \( \sigma' \) such that \( \mc K \backslash \{ \sigma' \} \) is not collapsible, \( \tau(\sigma') \ne \tau(\sigma_1)\) or \( \tau(\sigma') \ne \tau(\sigma_2) \), which can be treated as the previous point.

      As a result, for \( \sigma \) ( \( \mc K \backslash \{ \sigma \} \) is collapsible) and for \( \sigma' \) ( \( \mc K \backslash \{ \sigma' \} \) is not collapsible ) it holds that \( \tau (\sigma) = \tau (\sigma') \Rightarrow  \sigma \cap \sigma' = \{ v \}  \), so after collapses \( \mc K  \{ \sigma \} \) and \( \mc K \backslash \{ \sigma' \} \) we arrive at two identical simplicial complexes besides the hanging edge (\( \sigma'\) or \( \sigma\)) irrelevant for the weak collapsibility. Since a simplicial complex can not be simultaneously collapsible and not collapsible,  the question of weak collapsibility can always be resolved by the greedy algorithm which has polynomial complexity.
\end{proof}

\subsection*{Computational cost of the greedy algorithm}

The complexity of \Cref{algo:greedy} rests upon the precomputed \( \sigma \mapsto \Delta_\sigma \) structure that de-facto coincides with the boundary operator \( B_2 \) (assuming \( B_2 \) is stored as a sparse matrix, the adjacency structure describes its non-zero entries). Similarly, the initial \( F \) set can be computed alongside the construction of \( B_2 \) matrix. Another concession is needed for the complexity of the removal of elements from \( \Delta_{\sigma_i} \) and \( F \), which may vary from \( \mc O (1) \) on average up to guaranteed \( \log ( | \Delta_{\sigma_i} | ) \). As a result, given a pre-existing \( B_2 \) operator, \Cref{algo:greedy} runs linearly, \( \mc O ( m_1  )\), or almost linearly depending on the realisation, \( \mc O ( m_1 \log m_1 )\).

\section{ Preconditioning through the Subsampling of the \( 2 \)-Core }
\label{sec:heavy}
In \Cref{sec:cholesky} we have shown that the efficient computation of the Cholesky multiplier for \( \Lu 1 \) is complicated by arising cyclic terms. At the same time, the absence of the cyclic terms in the Schur complements is a characteristic of weakly collapsible complexes. 
Using this observation, in this section, we develop a Cholesky-like preconditioning scheme based on an efficient Cholesky multiplier for weakly collapsible complexes. 

First, we demonstrate that a weakly collapsible simplicial complex \(\mc K \) immediately yields an exact Cholesky decomposition for its up-Laplacian:
\begin{lemma}\label{lem:exact}
Assume the  2-skeleton simplicial complex \( \mc K \) is weakly collapsible through the collapsing sequence \( \Sigma \). Let \( \ds T \) be the corresponding sequence of maximal faces and let   \( P_\Sigma \) and \( P_{\ds T} \)  be the permutation matrices of the two sequences, i.e. such that  \( \left[ P_\Sigma \right]_{ij} = 1  \iff j = \sigma_i \), and similarly for $P_{\ds T}$. Then $C =  P_\Sigma B_2 P_{\ds T}$  is an exact Cholesky mulitplier for $P_\Sigma \Lu 1(\mc K) P_\Sigma^\top$, i.e.
      \( P_\Sigma \Lu 1 (\mc K) P_\Sigma^\top = C C^\top \).
\end{lemma}
\begin{proof}
      Note that the sequences \( \Sigma \) and \( \ds T \) (and the multiplication by the corresponding permutation matrices) impose only the reordering of \( \V 1 \) and \( \V 2\), respectively; after such reordering the \(i\)-th edge collapses the \(i\)-triangle. Hence, the first \((i-1) \) entries of the \( i \)-th column of the matrix \( B_2 \) ( \( \left[ B_2 \right]_{:, i} = \sqrt{w(t_i)} \b e_{t_i} \) )  are \( 0 \), otherwise one of the previous edges is not free. As a result, \( C \) is lower-triangular and by a direct computation one has \( C C^\top = P_\Sigma \Lu 1 (\mc K ) P_\Sigma^\top \).
\end{proof}

An arbitrary simplicial complex \( \mc K \) is generally not weakly collapsible (see \Cref{fig:2-core}). Specifically, weak collapsibility is a property of sparse simplicial complexes with the sparsity being measured by the number of triangles \(m_2 \) (in the weakly collapsible case \( m_2 < m_1 \) 
since each collapse at the edge eliminates exactly one triangle); hence, the removal of triangles from \( \V 2 \) can potentially destroy the 2-Core structure inside \( \mc K \) and make the complex weakly collapsible.

      With this observation, in order to find a cheap and effective preconditioner for \( \Lu 1(\mc K )\), one may search for a weakly collapsible subcomplex \( \mc L \subseteq \mc K \) and use its exact Cholesky multiplier \( C \) as a preconditioner for \( \mc K \).

Specifically, such subcomplex \( \mc L \) should satisfy the following properties:

      \begin{enumerate}[leftmargin = *, label=(\arabic*) ]
            \item it has the same set of nodes and edges, \( \mc V_0(\mc L) = \V 0 \) and \( \mc V_1(\mc L) = \V 1 \);  
            \item triangles in \( \mc L \) are subsampled, \( \mc V_2(\mc L ) \subseteq \V 2 \);
            \item \( \mc L \) is weakly collapsible through some  collapsing sequence \( \Sigma \) and corresponding sequence of maximal faces \( \ds T \);
            \item \( \mc L \) has the same 1-homology as \( \mc K \),  that is \( \ker L_1(\mc K ) =  \ker L_1 (\mc L ) \);
            \item the Cholesky multiplier \( C = P_\Sigma B_2 (\mc L ) P_{\ds T }\) improves the condition number of \( \Lu 1 (\mc K )\), namely $\kappa_+ ( C^+ P_\Sigma \Lu 1 (\mc K ) P_\Sigma^\top (C^+)^\top  ) \ll \kappa_+ ( \Lu 1 (\mc K) )$.
      \end{enumerate}

Let us comment on the conditions above. Conditions {(1)} and {(2)} are automatically met when a subcomplex \( \mc L \) is obtained from \( \mc K \) through the elimination of triangles. Condition {(3)} is a structural requirement on \( \mc L \) and can be guaranteed by the design of the subcomplex using the proposed \Cref{algo:heavy_subsampling}. Condition {(4)} guarantees that the preconditioning strategy is bijective, as we show in the next \Cref{rem:condition4}.  
Finally, condition {(5)} asks for a better condition number and is checked numerically in \Cref{sec:experiments}. However, whilst one can not guarantee improvement in preconditioning quality, we can provide an explicit formula for the condition number \( \kappa_+ ( C^+ P_\Sigma \Lu 1 (\mc K ) P_\Sigma^\top (C^+)^\top  ) \) assuming an arbitrary subcomplex forms the preconditioner \( C \). We provide such a formula in the next \Cref{subsec:subcomplex} and discuss how to use it to construct a preconditioner via heavy collapsible subcomplex.

\begin{lemma}[On the conservation of the 1-homology red and condition {(4)}]\label{rem:condition4}
    For any  \( \mc L \) subcomplex of \( \mc K \), the following statements about 1-homology hold:
    \begin{enumerate}[label=(\roman*)]
        \item if  \( \mc L \) satisfies conditions {(1)} and {(2)} then it can only extend the kernel, i.e.\ we have   \( \ker L_1 (\mc K) \subseteq \ker L_1 ( \mc L )\);
        \item if additionally \( \mc L \) satisfies condition {(4)}, then the 
         preconditioning scheme is bijective, in the sense that the solution to the original least square problem $\min_{\b x}\|\Lu 1 (\mc K ) \b x -\b f\|$ and the one preconditioned with the Cholesky factor of \( \mc L \) coincide.
    \end{enumerate}
\end{lemma}
\begin{proof}
    For (i) it is sufficient to note that the elimination of the triangle \(t \in \V 2 \) lifts the restriction \( \b e_t^\top \b x = 0\) for \( \b x \in \ker L_1 (\mc K)\); hence, if \( \b x \in \ker L_1( \mc K )\), then \( \b x \in \ker L_1(\mc L)\). For (ii),  note that bijection between the systems \( \Lu k \b x = \b f \) and \( \left( C^+  \Lu k (C^+)^\top \right) C^\top \b x = C^+  \b f  \) is guaranteed by \( \ker C^\top = \ker \Lu k  = \ker B_{k+1}^\top \) (assuming \( \b x \perp \ker \Lu k\)). Then, by
      the spectral inheritance principle, {\cite[Thm.~2.7]{guglielmi2023quantifying}}, 
      \( \ker \Lu k (\mc X ) = \ker L_k (\mc X ) \oplus B_{k}^\top  \cdot \im \Lu {k-1} \).  The second part, \( B_{k}^\top  \cdot \im \Lu {k-1} \), consists of the action of \( B_{k}^\top \) on non-zero related eigenvectors of \( \Lu {k-1}\) and is not dependent on \( \V {k+1} \) (triangles, in case \( k =1 \)), hence remains preserved in the subcomplex \( \mc L\). Since by condition {(4)} \( \ker L_1(\mc K ) = \ker L_1(\mc L ) \), the same statement holds for up-Laplacians, \( \ker \Lu 1(\mc K ) = \ker \Lu 1 (\mc L) \). Since \( C \) is an exact Cholesky multiplier for \( \Lu 1 (\mc L ) \), \( \ker \Lu 1 (\mc L ) = \ker C^\top\) and \( \ker \Lu 1 (\mc K ) = \ker B_2^\top \) yielding \( \ker C^\top = \ker B_2^\top \) and bijectivity. 
\end{proof}

\subsection{Preconditioning quality by the Subcomplex}\label{subsec:subcomplex}

      Note that subcomplex \( \mc L \) is fully defined by the subset \( \ds T \) of subsampled triangles, \( \ds T \subset \V 2 \), so  \(\mc L = \V 0 \cup \V 1 \cup \ds T \)). We introduce the following matrix notation corresponding to the subsampling: 

\begin{definition}[Subsampling matrix]
     Let \( \ds T \) be a subset of triangles, \( \ds T \subset \V 2 \), then \( \Pi \) is a subsampling matrix if 
      \begin{itemize}[leftmargin = 45pt]
            \item \( \Pi \) is diagonal, \( \Pi \in \ds R^{m_2 \times m_2}\);
            \item \( ( \Pi )_{ii} = 1  \iff i \in \ds T \); otherwise, \( ( \Pi )_{ii} = 0 \).
      \end{itemize}
\end{definition}

Assuming the subset \( \ds T \) of subsampled triangles (or, equivalently, the subsampling matrix \( \Pi \)) is given, one needs only the triangle weight matrix \( \wh W\) in order to obtain \( \Lu 1(\mc L )\) and the corresponding Cholesky multiplier \( C \). 

Generally speaking, (squared) weights \( \wh W_2^2 \) of sampled triangles \( \ds T \) may differ from the original weights \( W_2^2 \). Let \( \wh W_2^2 = W_2^2 + \Delta W_2\), where \( \Delta W_2 \) is still diagonal, but entries are not necessarily positive. Then one can formulate the question of the optimal weight redistribution as the optimization problem:
\begin{equation*}
      \min_{ \Delta W_2 } \left\| \Lu 1(\mc L) - \Lu 1 (\mc K) \right\| = \min_{ \Delta W_2 } \left\| B_2 \left[ \Pi( W_2^2 + \Delta W_2 ) \Pi - W_2^2 \right] B_2^\top \right\|
\end{equation*}
Here and only here, since we manipulate weights, we use the unweighted \( B_2 \) matrix so one can have explicit access to the weight matrix \( W_2 \).

\begin{lemma}[Optimal weight choice for the subcomplex]\label{lemma:subsample_weight}
      Let \( \mc L \) be subcomplex of \( \mc K \) with fixed corresponding subsampling matrix \( \Pi \). Then the optimal weight perturbation for the subsampled triangles is  
      \begin{equation*}
            \Delta W_2 \equiv 0,
      \end{equation*}
      so the best choice of weights of subcomplex is \( \wh W_2 = W_2 \Pi \).
\end{lemma}
\begin{proof}
Let \( \Delta W_2 = \Delta W_2 (t )\) where \(t \) is a time parametrization; we can compute the gradient \( \nabla_{\Delta W_2} \sigma_1 \left( \Lu 1(\mc L) - \Lu 1 (\mc K) \right) \) through the  derivative \( \frac{d}{dt} \sigma_1 \left(  \Lu 1(\mc L) - \Lu 1 (\mc K) \right)  \):
\begin{equation*}
      \begin{aligned}
            \dot{\sigma}_1 & = \b x^\top B_2 \Pi \dot{ \Delta W_2 } \Pi B_2^\top \b x = \left\langle B_2 \Pi \dot{ \Delta W_2 } \Pi B_2^\top , \, \b x \b x^\top \right\rangle = \mathrm{Tr} \left( B_2 \Pi \dot{ \Delta W_2 } \Pi B_2^\top \b x \b x^\top \right) = \\ 
            & = \left\langle \Pi B_2^\top \b x \b x^\top B_2 \Pi , \, \dot{ \Delta W_2 } \right\rangle = \left\langle \nabla_{\Delta W_2} \sigma_1, \dot{\Delta W_2 } \right\rangle 
      \end{aligned}
\end{equation*}
By projecting onto the diagonal structure of the weight perturbation, 
\begin{equation*}
      \nabla_{\diag \Delta W_2 } \, \sigma_1 = \diag \left( \Pi B_2^\top \b x \b x^\top B_2 \Pi  \right).      
\end{equation*}
Note that \( \diag \left( \Pi B_2^\top \b x \b x^\top B_2 \Pi  \right)_{ii} = |  \Pi B_2^\top \b x |^2_i \); then the stationary point is characterized by \(  \Pi B_2^\top \b x = 0 \iff \b x \in \ker \Lu 1\). The latter is impossible since \( \b x \) is the eigenvector corresponding to the largest eigenvalue; hence, since \( \Pi (W_2^2 + \Delta W_2 ) \Pi \ne  W_2^2 \), the optimal perturbation is \( \Delta W_2 \equiv 0 \). 
\end{proof}

      We have established the optimal choice of the weights provided the subsampling matrix \( \Pi \). As a result, assuming \( B_2 \) is weighted, optimal \( \Lu 1 (\mc L) = B_2 \Pi B_2^\top \). Now we proceed to characterize the quality of the preconditioning in terms of the matrix \( \Pi \) in \Cref{thm:cond_weight}, starting with a necessary technical relation between \(\im B_2^\top\) and \( \ker \Pi\).

\begin{remark}
      \label{rem:ker_im}
     Knowing the optimal weight for sampled triangles from \Cref{lemma:subsample_weight}, one needs to preserve the kernel of subsampled Laplacian
      \[
            \ker \left( B_2 \Pi B_2^\top \right) = \ker \left(  B_2 B_2^\top \right) 
      \]
      to form a correct preconditioner \( C \). 
Since we have \( \Pi = \Pi^2 \) and \( \ker \Lu 1 = \ker B_2^\top \), then \( \ker \left(   B_2 \Pi B_2^\top \right) = \ker \left( \Pi B_2^\top \right) \). Additionally, \( \ker B_2^\top \subseteq \ker \left( \Pi B_2^\top \right)  \), so \(  \ker \left(  B_2 \Pi B_2^\top \right) \ne \ker \left(  B_2 B_2^\top \right) \iff \) there exists \( \b y \in \im B_2^\top\) such that \( B_2^\top \b y \ne 0 \) and \( B_2^\top \b y \in \ker \Pi\). Then in order to preserve the kernel, one needs 
\( \im B_2^\top \cap \ker \Pi = \{ 0 \} \). 
\end{remark}

\begin{theorem}[Conditioning by the Subcomplex]
      \label{thm:cond_weight}
      Let \( \mc L \) be a weakly collapsible subcomplex of \( \mc K \) defined by the subsampling matrix \( \Pi \) and let \( C \) be a Cholesky multiplier of \( \Lu 1 (\mc L) \) defined as in \Cref{lem:exact}. Then
             the conditioning of the symmetrically preconditioned \( \Lu 1 \) is given by:
            \begin{equation*}
                  \kappa_+ \left( C^+ P_\Sigma \Lu 1 P_\Sigma^\top \left( C^{+}\right)^\top \right) = \left( \kappa_+ \left( \left( S_1 V_1^\top \Pi \right)^+ S_1 \right) \right)^2 = \left( \kappa_+ ( \Pi V_1 ) \right)^2, 
            \end{equation*}
            where \( V_1 \) forms the orthonormal basis on \( \im B_2^\top \).
\end{theorem}
\begin{proof}
By \Cref{lemma:subsample_weight}, \(W_2( \mc L ) = \Pi W_2\); then let us consider the lower-triangular preconditioner  \( C =  P_\Sigma B_2 \Pi P_{\ds T}  \) for \( P_\Sigma \Lu 1 P_\Sigma^\top  \); then the preconditioned matrix is given by:
\begin{equation*}
      \begin{aligned}
            C^+ \left( P_\Sigma \Lu 1 P_\Sigma^\top \right) (C^+)^{\top} & = \left( P_\Sigma B_2 \Pi P_{\ds T} \right)^+ \left( P_\Sigma \Lu 1 P_\Sigma^\top \right) \left(\left( P_\Sigma B_2 \Pi P_{\ds T} \right)^+\right)^{\top} = \\
            & = P_{\ds T}^\top \left( B_2 \Pi \right)^+ 
              \Lu 1 \left(\left( B_2 \Pi  \right)^+\right)^{\top} P_{\ds T}
      \end{aligned}
\end{equation*}
Note that \( P_{ \ds T } \) is unitary, so \( \kappa_+ ( P_{ \ds T } X P_{ \ds T }^\top ) = \kappa_+ ( X ) \), hence the principle matrix is \( \left( B_2 \Pi \right)^+ \Lu 1 \left(\left( B_2 \Pi \right)^+\right)^{\top } = \left( B_2 \Pi \right)^+ B_2  B_2^\top \left(\left( B_2 \Pi  \right)^+\right)^{\top}  \).
Since  \( \kappa_+ ( X^\top X ) = \kappa_+^2( X ) \), then in fact one needs to consider $\kappa_+ \left( \left( B_2 \Pi \right)^+ ( B_2  ) \right)$. 
Let us consider the SVD-decomposition for \(  B_2 = U S V^\top \); more precisely,
\begin{equation*}
      B_2  = U S V^\top = \begin{mt} U_1 & U_2 \end{mt} \begin{mt} S_1 & 0 \\ 0 & 0 \end{mt} \begin{mt} V_1^\top \\ V_2^\top \end{mt} = U_1 S_1 V_1^\top
\end{equation*}
where \( S_1 \) is a diagonal invertible matrix. Note that \( U \) and \( U_1 \) have orthonormal columns and \( S_1 \) is diagonal and invertible, so
\begin{equation*}
      (  B_2 \Pi )^+  B_2 = \left( S V^\top \Pi \right)^+ S V^\top = \left( S_1 V_1^\top \Pi \right)^+ S_1 V_1^\top 
\end{equation*}
By the definition of the condition number \( \kappa_+ \), one needs to compute \( \sigma^+_{\min}\) and \( \sigma^+_{\max}\) where:
\begin{equation*}
      \sigma^+_{\min} = \underset{ \b x \perp \ker \left( \left( S_1 V_1^\top \Pi \right)^+ S_1 V_1^\top \right) }{\min} \frac{ \left\| \left( S_1 V_1^\top \Pi \right)^+ S_1 V_1^\top \b x \right\| }{ \| \b x \| }
\end{equation*}
and analogously for \( \sigma^+_{\max} \). Here we prove the result for \( \sigma^+_{\min} \). The proof transfers unchanged to \( \sigma^+_{\max}\) by changing \( \min \) to \(\max\).

Note that \( \im B_2^\top = \im V_1 = \im V_1 S_1  \), so by \Cref{rem:ker_im}, \( \ker \Pi \cap \im V_1 S_1 = \{ 0 \} \), hence \( \ker \Pi V_1 S_1 = \ker V_1 S_1 \). Since \(\ker V_1 S_1 \cap \im S_1 V_1^\top = \{0\} \), one gets \(\ker \Pi V_1 S_1 \cap \im S_1 V_1^\top = \{0\} \). By the properties of the pseudo-inverse we have that $ \ker \Pi V_1 S_1 = \ker \left( S_1 V_1^\top \Pi \right)^\top = \ker \left( S_1 V_1^\top \Pi \right)^+ $; as a result, \( \ker \left( \left( S_1 V_1^\top \Pi \right)^+ S_1 V_1^\top \right) = \ker S_1 V_1^\top \). Since \( S_1 \) is invertible, \( \ker \left( \left( S_1 V_1^\top \Pi \right)^+ S_1 V_1^\top \right) = \ker V_1^\top \).
 
For \( \b x \in \ker V_1^\top \Rightarrow \b x \in \im V_1 \), so \( \b x = V_1 \b y \). Since \( V_1^\top V_1 = I \), \( \| \b x \| = \| V_1 \b y \| \) and:
\begin{equation*}
      \sigma^+_{\min} =  \underset{ \b y }{\min } \frac{ \left\| \left( S_1 V_1^\top \Pi \right)^+ S_1 \b y  \right\|  }{ \| \b y \| } \underset{ \b z = S_1 \b y }{ = } \underset{ \b z }{\min} \frac{ \left\| \left( S_1 V_1^\top \Pi \right)^+ \b z \right\|  }{ \left\| S_1^{-1} \b z \right\| }
\end{equation*}
Note that \( \b v = \left( S_1 V_1^\top \Pi \right)^+ \b z \iff \begin{cases} S_1 V_1^\top \Pi \b v = \b z \\ \b v \perp \ker S_1 V_1^\top \Pi  \end{cases} \) and \( \ker S_1 V_1^\top \Pi = \ker V_1^\top \Pi \), so:
\begin{equation*}
      \sigma^+_{\min } = \underset{ \b v \perp \ker V_1^\top \Pi }{\min } \frac{ \| \b v \| }{ \| V_1^\top \Pi \b v \|  }
\end{equation*}
Hence \( \kappa_+ \left( C^+ P_\Sigma \Lu 1 P_\Sigma^\top (C^+)^{ \top} \right) = \kappa_+^2 ( V_1^\top \Pi ) = \kappa_+^2 ( \Pi V_1 )\).

\end{proof}

\subsection{Algorithm: Preconditioner via Heavy Collapsible Subcomplex}

      Following \Cref{thm:cond_weight}, note that
    \( C \) is a perfect preconditioner for \( P_\Sigma \Lu 1 P_\Sigma^\top\)
      if \( \Pi = I \), since \( \kappa_+(V_1) = 1 \) and we would compute the exact Cholesky decomposition. However, this is computationally prohibitive. 
      Thus, it is natural to try to find a sparser \( \Pi \) so that \( \Pi V_1 \) is as close to \( V_1 \) as possible. Multiplication by \( \Pi \) cancels rows in \( V_1 \) corresponding to the eliminated triangles; at the same time, rows in \( V_1 \) are scaled by the weights of the triangles since \( \mathrm{span} V_1 = W_2 \im B_2^\top \), so we expect to  \( \Pi \) eliminating triangles with lowest weights to be a good choice.

      Based on this observation and \cref{thm:cond_weight}, we provide an algorithm for preconditioning \( \Lu 1(\mc K )\) which aims to eliminate triangles of the lowest weight thus constructing what we call a heavy weakly collapsible subcomplex \( \mc L \) with largest possible total weight of triangles. The exact Cholesky multiplier \( C \) of \( \Lu 1(\mc L ) \) is cheap to compute and is used as a preconditioner for \( \Lu 1(\mc K )\).

The proposed \Cref{algo:heavy_subsampling} works as follows: start with an empty subcomplex \( \mc L \); then, at each step try to extend \( \mc L \) with the heaviest unconsidered triangle \( t\): \( \mc L \to \mc L \cup \{ t \} \) -- here the extension includes the addition of the triangle \(t\) with all its vertices and edges to the complex \(\mc L\). If the extension \( \mc L \cup \{ t \} \) is weakly collapsible, it is accepted as the new \( \mc L \), otherwise \( t \) is rejected; in either case the triangle \(t\) is removed from the list of unconsidered triangles, i.e. $t$ is not considered for a second time.

\begin{algorithm}[h]
      \caption{ \texttt{HEAVY\_SUBCOMPLEX}\( (\mc K, W_2) \): construction a heavy collapsible subcomplex
      \label{algo:heavy_subsampling}}
      \begin{algorithmic}[1]
            \Require the original complex \( \mc K \), weight matrix \( W_2 \)
            \State \( \mc L  \gets \vn, \, \ds T \gets \vn  \) \Comment{initial empty subcomplex}
            \While{ there is unprocessed triangle in \( \V 2 \) }
                  \State \( t \gets \texttt{nextHeaviestTriangle}(\mc K, W_2) \) \Comment{e.g. iterate through \( \V 2 \) sorted by weight}
                  \If{ \(\mc L \cup \{ t \} \) is weakly collapsible } \Comment{\small run \texttt{GREEDY\_COLLAPSE}\( (\mc L \cup \{ t\} ) \) (\Cref{algo:greedy})}
                        \State \( \mc L \gets \mc L \cup \{ t \},\, \ds T \gets [ \ds T \; t ]  \) \Comment{ extend \( \mc L \) by \( t \)}
                  \EndIf
            \EndWhile
             \State \Return \( \mc L, \, \ds T, \, \Sigma \) \Comment{here \( \Sigma \) is the collapsing sequence of \( \mc L \) }
      \end{algorithmic}
\end{algorithm}

\begin{remark}
      We show next that a subcomplex \( \mc L \) sampled with \Cref{algo:heavy_subsampling} satisfies properties (1)--(4) above: indeed, \( \V 0 = \mc V_0(\mc L)\), \( \V 1 = \mc V_1(\mc L) \) and \( \mc L \) is weakly collapsible by construction. It is less trivial to show that the subsampling \( \mc L \) does not increase the dimensionality of \(1\)-homology (see \Cref{rem:condition4}).

      Assuming the opposite, the subcomplex \( \mc L \) cannot have any additional \(1\)-di\-men\-sio\-nal holes in the form smallest-by-inclusion cycles of more than \( 3 \) edges: since this cycle is not present in \( \mc K \), it is ``covered'' by at least one triangle \( t \) which necessarily has a free edge, so \( \mc L \) can be extended by \(t\) and remain weakly collapsible. Alternatively, if the only additional hole corresponds to the triangle \( t \) not present in \( \mc L \); then, reminiscent of the proof for \Cref{thm:poly}, let us consider the minimal by inclusion simplicial complex \( \mc K \) for which it happens. Then the only free edges in \( \mc L \)  are the edges of \( t \), otherwise \( \mc K \) is not minimal. At the same time, in such setups \( t\) is not registered as a hole since it is an outer boundary of the complex \( \mc L \), e.g. consider the exclusion of exactly one triangle in the tetrahedron case, \Cref{fig:2-core}, which proves that \( \mc L \) cannot extend the 1-homology of \( \mc K \). 
\end{remark}

The complexity of \Cref{algo:heavy_subsampling} is \( \mc O( m_1 m_2 ) \) at worst which could be considered comparatively slow: the algorithm passes through every triangle in \( \V 2 \) and performs collapsibility check via \Cref{algo:greedy} on \( \mc L \) which never has more than \( m_1 \) triangles since it is weakly collapsible. Note that 
 \Cref{algo:heavy_subsampling} and \Cref{thm:cond_weight} do not depend on \( \mc K \) being a 2-Core; moreover, the collapsible part of a generic \( \mc K \) is necessarily included in the subcomplex \( \mc L \) produced by \Cref{algo:heavy_subsampling}. Hence a prior pass of \texttt{GREEDY\_COLLAPSE}\((\mc K)\) reduces the complex to a smaller 2-Core \(\mc K'\) with faster \texttt{HEAVY\_SUBCOMPLEX}\((\mc K', W_2) \) since \( \mc V_1(\mc K') \subset \V 1 \) and \( \mc V_2(\mc K') \subset \V 2 \).

We summarise the whole procedure for computing the preconditioner next:
\begin{itemize}
    \item reduce a generic simplicial complex \( \mc K \) to a 2-Core \( \mc K'\) through the collapsing sequence \( \Sigma_1 \) and the corresponding sequence of triangles \( \ds T_1 \) through the greedy \Cref{algo:greedy};
    \item form a heavy weakly connected subcomplex \( \mc L \) from \( \mc K'\) with the collapsing sequence \( \Sigma_2 \) and the corresponding sequence of triangles \( \ds T_2 \) using \Cref{algo:heavy_subsampling};
    \item form the preconditioner \( C \)  by permuting and subsampling \( B_2 \) using the subset of triangles \( \ds T =  \ds T_1 \cup \ds T_2 \) (that determines the subsampling matrix \( \Pi \)) and the associated collapsing sequence \( \Sigma =  \Sigma_1 \cup \Sigma_2  \), via \( C = P_\Sigma B_2 \Pi P_{\ds T} \).
\end{itemize}

\begin{figure}[htbp]
\centering
\scalebox{0.8}{
	\begin{tikzpicture}
		\node[align=center] at (0, -0.75){ original \\ complex \( \mc K \) };
		\node[align=center] at (4, -0.75){2-Core \( \mc K' \) \\ \( \mc K' \subset \mc K  \) };
		\node[align=center, draw=red, inner sep = 1pt] at (12, -0.75){weakly collapsible \\ heavy subcomplex \( \mc L \)};
		\draw[-latex, line width=1.5] (5, -0.75)  -- node[midway, below, align=center]{\small \texttt{HEAVY\_SUBCOMPLEX}\((\mc K', W_2) \) \\ \Cref{algo:heavy_subsampling} } (10, -0.75);
		\draw[-latex, line width=1.5] (1.15, -0.75) -- (3, -0.75);
		\node[align=center, draw=red, inner sep=1pt] at (4.0, 0.75) {\small \texttt{GREEDY\_COLLAPSE}(\( \mc K \)) \\ \small \Cref{algo:greedy} };
            \node[red, align = center,] at (10, 1.25) {projection matrix \( \Pi \) \\ preconditioner \( C \)};
            \draw[-latex, red, dashed, thick] (5.675, 0.45)  -| (10, 0.75);
            \draw[-, red, dashed, thick] (12, -0.25) |- (10, 0.45 );
            \draw[-, red, dashed, thick] (1.75, -0.6) |- (2.5, 0.45 );
	\end{tikzpicture}}
	\caption{ The scheme of the simplicial complex transformation: from the original \( \mc K \) to the heavy weakly collapsible subcomplex \( \mc L \). \label{fig:scheme}}
\end{figure}
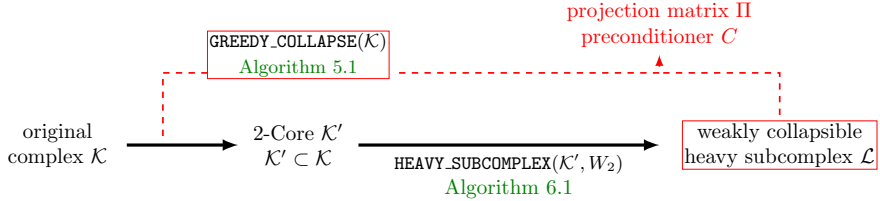

We refer to the preconditioner built in this way (see also \Cref{fig:scheme}) as a \emph{heavy collapsible subcomplex} preconditioner (\algname-preconditioner).

\section{Numerical experiments} \label{sec:experiments}
We present here a number of numerical experiments to validate the performance of the proposed preconditioning strategy. All the experiments are run using {\tt julia} on Apple M1 CPU and can be reproduced with the code available at \url{https://github.com/COMPiLELab/HeCS}.

\subsection{Conjugate Gradient Least Square Method} 
The preconditioning performance for the least square problem
\( \min_{\b x \perp \ker \Lu 1}\| \Lu 1 \b x - \b f  \| \)
is measured on the \emph{conjugate gradient least square method} (CGLS), \cite{bjorck1998stability, hestenes1952methods}. The method requires computing \texttt{matvec} operations for the matrix \( \Lu 1 \) and its preconditioned alternatives; CGLS converges as $(\sqrt{ \kappa_+( A ) } - 1)/(\sqrt{ \kappa_+( A ) } + 1 )$ and we run it until the infinity norm of the residual \( \b r_i = \Lu 1 \b x_i - \b f\) falls below a given threshold, i.e.~\( \| \b r_i \|_\infty \le \epsilon \).

\subsection{Shifted incomplete Cholesky preconditioner}
$\Lu 1$ is a singular matrix. Assuming \( U \) is an orthogonal basis of \( \ker \Lu 1\), one can move to \( \Lu 1 \to \Lu 1 + \alpha U U^\top \) which can be preconditioned by non-singular methods. Specifically, we use \( C_\alpha = \texttt{ichol} (  \Lu 1 + \alpha U U^\top  ) \) to compare with the \algname~preconditioner, \Cref{fig:scheme}.

Calculating such a shift requires efficiently computing \( \ker \Lu 1 \), which in principle has a complexity comparable to the original system. On the other hand, in our setting, given the spectral inheritance principle (see e.g. \cite[Thm. 2.7]{guglielmi2023quantifying}), an orthogonal basis \( U \) can be formed directly using the vectors \( B_1^\top \b x \), \( \b x \in (\bf 1 )^\perp \), when \( \mc K \) has trivial 0- and 1-homologies (i.e.\ it is formed by one connected component, \( \ker L_0 =\mathrm{span} \{ \bf 1 \} \), and has no 1-dimensional holes, \( \ker L_1 = 0 \)). 

Note that the \algname{} preconditioner instead works without requiring any triviality of the topology of the complex.

\subsection{Problem setting: Enriched triangulation as a simplicial complex} 

To illustrate the behaviour of the preconditioned system \( C^+ P_\Sigma \Lu 1 P_\Sigma^\top (C^+)^{\top} \), we consider a sparse simplicial complex \( \mc K \), i.e.\ we assume \( m_2 = \mc O (m_1 \ln m_1 ) \). Note that the developed routine can be applied in denser cases, although one can expect a certain loss of efficiency. 
To generate problem settings within this range, \( \mc K \) is synthesized as an enriched triangulation of \( N \) points on the unit square with a prescribed edge sparsity pattern \( \nu \) as follows:
\begin{enumerate}[label=(\arabic*),leftmargin=*]
      \item \( \V 0 \) is formed by the corners of the unit square and \( (N-4)\) points sampled uniformly at random from \( U \left( [0, 1]^2 \right) \);
      \item the Delaunay triangulation of \( \V 0 \) is computed; all edges and \(3\)-cliques of the produced graph are included in \( \V 1 \) and \( \V 2 \) respectively;
      \item \( d \ge 0 \) edges (excluding the outer boundary) are chosen at random and eliminated from \( \V 1 \); triangles adjacent to the chosen edges are eliminated from \( \V 2 \). As result, produced complex \( \mc K \) has a non-trivial \(1\)-homology;
      \item the  sparsity pattern is defined as \( \nu = m_2 / q(m_1) \)  where  \( q(m_1) = \mc O(m_1 \ln m_1)\) is the highest density of triangles for the sparse case; 
      additional edges on \( \V 0 \) are added to \( \V 1 \) alongside with new appearing \( 3 \)-cliques to reach the target \( m_2 / q(m_1) \) value. The initial sparsity pattern of the triangulation is denoted by \( \nu_\Delta \).
\end{enumerate}

\subsection{Heavy subcomplex and triangle weight profile} 
\Cref{algo:heavy_subsampling}, aims to build a heavy weakly collapsible subcomplex \( \mc L \) such that the total weight of triangles in \( \mc L \) is close to the total weight of triangles in the original complex \( \mc K \). At the same time, the number of triangles in \( \mc L \) is limited, \( m_2(\mc L) < m_1(\mc L) = m_1(\mc K) \), due to the weak collapsibility, while the number of triangles in \( \mc K \) can go up to \( q(m_1(\mc K )) \).
Hence, the quality of the preconditioner 
is determined by the triangle weight distribution \( w(\cdot )\) on \( \V 2 \): 
namely, if \( w_2(t) \) are distributed uniformly and independently, the quality of the preconditioning falls rapidly after \( \nu > \nu_\Delta \) and, at the same time, the original matrix \( \Lu 1 \) remains well-conditioned in such configurations. Instead, we observe that unbalanced weight distributions lead to ill-conditioned \( \Lu 1 \) and, thus, here we consider the following two situations:
\begin{itemize}[leftmargin = *]
      \item the weights of triangles are random variables that are \emph{independent} of each other and distributed as heavy-tailed Cauchy distributions or bi-modal Gaussian distribution \( \mc N( 1, \sigma_1 ) + \mc N( 1/3, \sigma_2 ) \). This way we generate a sufficiently large number of heavy triangles and a cluster of reducible triangles with small~\( w_2(t) \);
      \item the weights of triangles are \emph{topologically dependent}, that is \( w_2(t)  \) is a function of the weights of the neighboring triangles. 
      A way to implement this dependence is to set  $ w_2(t) = f(w_1(e_1), w_1(e_2), w_1(e_3))$ for the triangle \( t = (e_1, e_2, e_3 ) \). Two well-known common choices of \( f \) are the min-rule $ w_2(t) = \min(w_1(e_1), w_1(e_2), w_1(e_3))$, \cite{guglielmi2023quantifying,lee2019coidentification}, and the product rule \( w_2(t) = w_1(e_1)  w_1(e_2)  w_1(e_3)\), \cite{chen2021decomposition, chen2021helmholtzian}. In this way, the edge weight profile \( w_1(\cdot)\) on \( \V 1\) is transformed to an unbalanced distribution \( w_2(t)\).
\end{itemize}

\subsection{Timings} 
One needs to separately discuss the time cost of the computation of the preconditioning operator (as a tuplet \( C\) and the permutation \( P_\Sigma \)) and of the \texttt{matvec} computation for the preconditioned operator \( C^+ P_\Sigma \Lu 1 P_\Sigma^\top (C^+)^{\top} \). 

Note that \texttt{matvec} of \( \Lu 1 \) has the complexity of the number of non-zero elements, \( \mc O( m_2 ) \); the \algname~preconditioner \( C \) has \(\le  3 m_1\) non-zero elments and lower triangular structure, so \texttt{matvec}s of either \( C^+ \) and \( (C^+)^{\top}\), as well as the permutation matrix \( P_\Sigma \), have  complexity  \( \mc O(m_1)\). Hence, the complexity of each preconditioned CGLS iteration is \( \mc O(m_2 + m_1)\), as opposed to the original \( \mc O(m_2)\); asymptotically one expects \( m_2 \gg m_1 \), so the preconditioning scheme is efficient. On the other hand, the shifted \texttt{ichol} preconditioner \( C_\alpha\) loses the sparsity due to the shift; as a result, the application of \( C_\alpha^+ \) costs \( \mc O(m_0 m_1 + m_2) \)  since \( \mathrm{rank}\, U = \mc O(m_0) \) in all considered scenarios. 

In  \Cref{fig:time_single} we compare the performance of the two preconditioners for the enriched triangulation on \(m_0=32\) vertices and a varying number of edges \(m_1\): the cost of one CGLS iteration for \algname~preconditioner is higher than the original system but asymptotically approaches the \texttt{matvec} cost of the unpreconditioned system, whilst the \texttt{ichol}-preconditioner \( C_\alpha \) is an order larger.

\begin{figure}
     \centering
     \begin{subfigure}[b]{0.49\textwidth}
         \centering
         \hspace{-3em}
         \includegraphics[width=1.0\columnwidth]{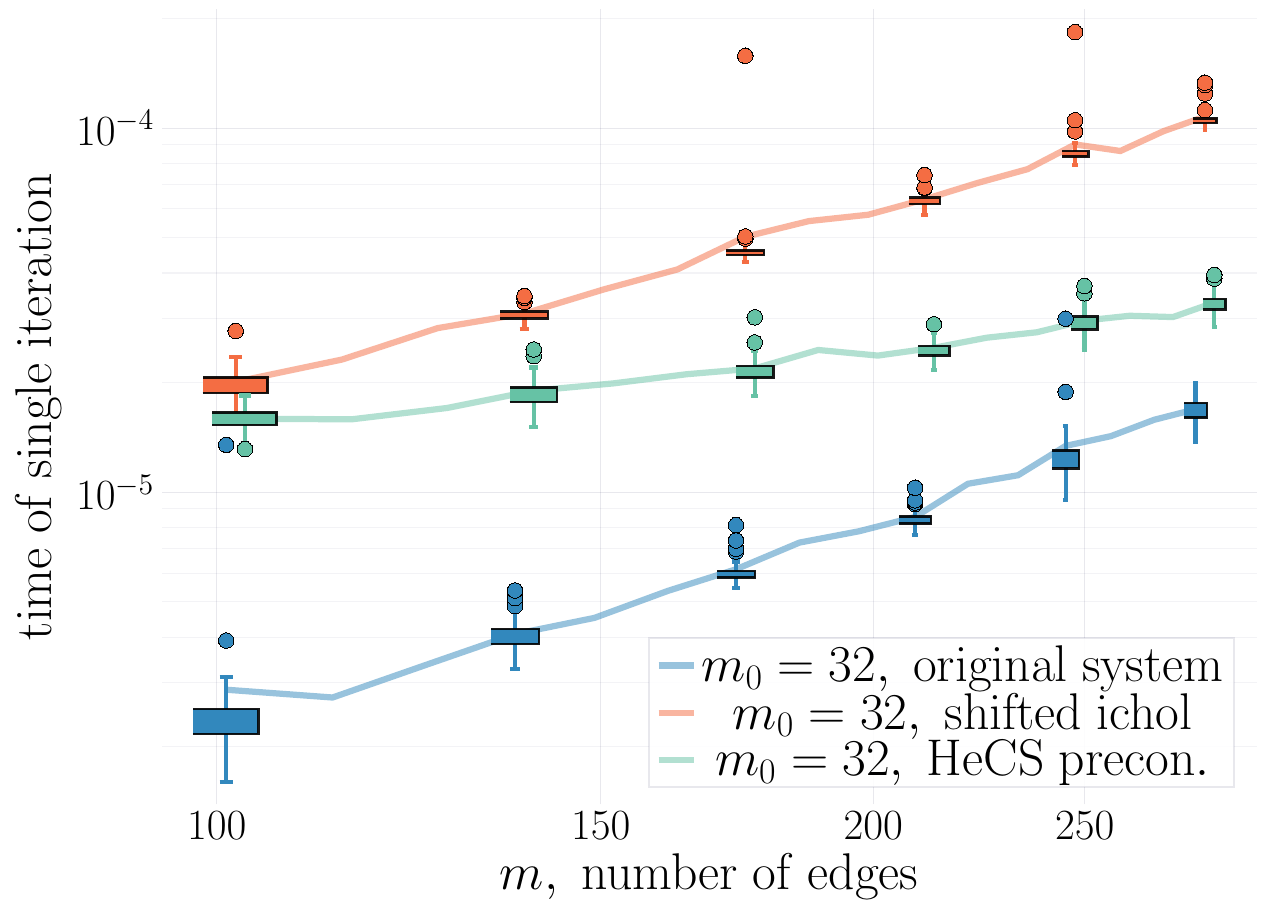}
         \caption{\small Single iteration timing: the average time of \texttt{matvec} computation for the original system (blue), shifted \texttt{ichol} (orange) and \algname~ preconditioner (green).}
         \label{fig:time_single}
     \end{subfigure}
     \hfill
     \begin{subfigure}[b]{0.49\textwidth}
         \centering
         \hspace{-2em}
         \includegraphics[width=1.0\columnwidth]{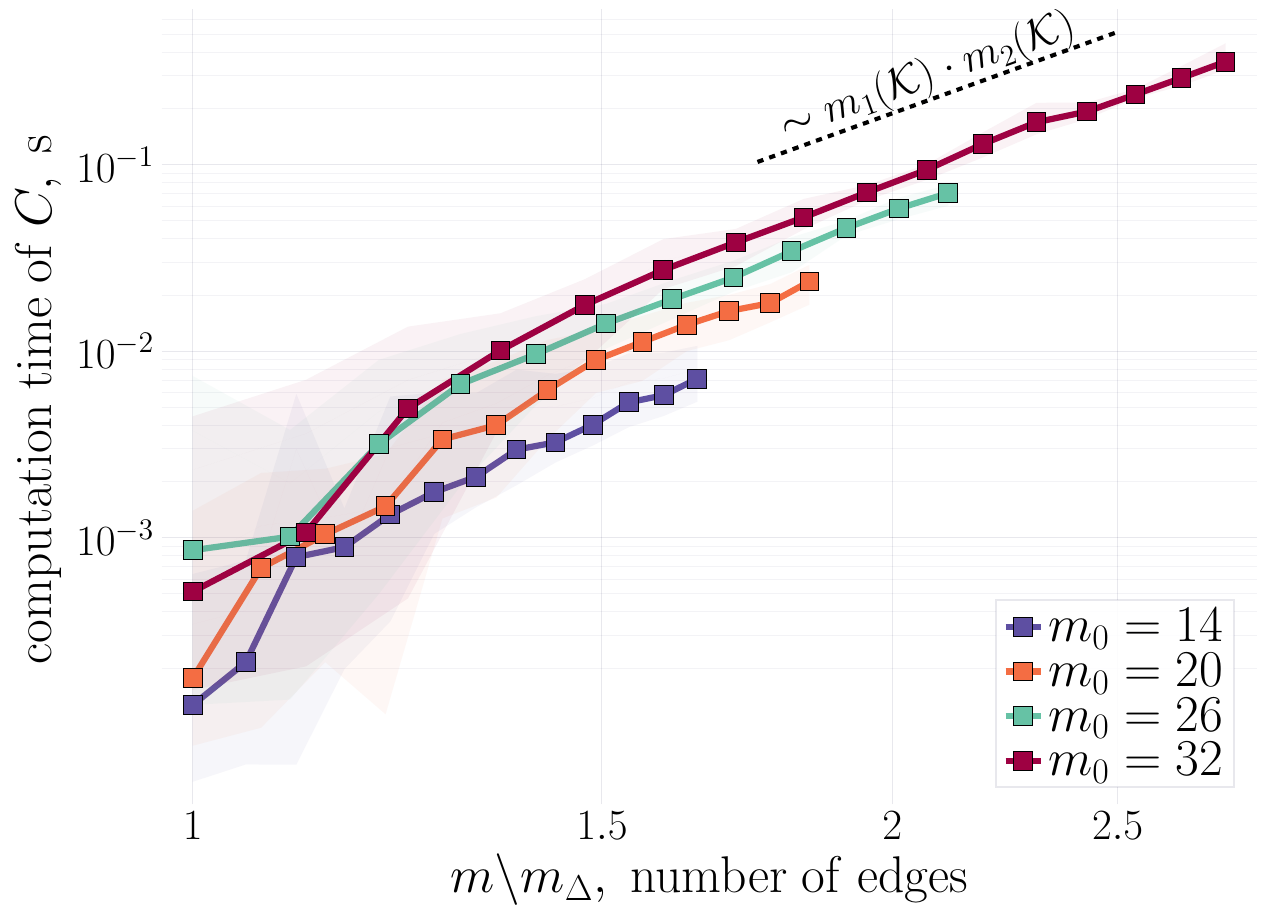}
         \caption{\small Computation time for the heavy subcomplex preconditioner in case of enriched triangulations on \( m_0 \) vertices }
          \label{fig:time_precon}
     \end{subfigure}
        \caption{Timings of \algname-perconditioner}
        \label{fig:timings}
\end{figure}

Additionally, we demonstrate the time complexity for the computation of the \algname~preconditioner  for enriched triangulation on \( m_0 = 14, 20, 26, 32 \) vertices and varying total edge number \( m_1 \), \Cref{fig:time_precon} (here \( m_\Delta \) denotes the number of edges in the initial triangulation); as \( \mc K \) becomes a denser simplicial complex, \texttt{HEAVY\_COMPLEX}(\( \mc K \)) follows the expected complexity \( \mc O( m_1 \cdot m_2 )\). 
 In comparison with the complexity of Cholesky decomposition which is \( \mc O(m_1^3) \), in the sparse case \( m_2 = \mc O ( m_1 \ln m_1 ) \) and the overall cost of \algname{} computation is always upper-bounded by \( \mc O( m_1^2 \ln(m_1) )\).

\subsection{Performance of the preconditioner}
We demonstrate the quality of \algname~preconditioner for enriched triangulations on \(m_0\) vertices with \(d = 2\) initially eliminated edges and for varying total number of edges \( m_ 1\) such that the initial sparsity \( \faktor{ m_2 }{ q(m_1) } \) is increased until the induced number of triangles \(m_2\) reaches \(q(m_1) = 9 C^2 m_1 \ln (4 m_1) \) with \( C = \frac{1}{2} \). This quantity is chosen in accordance to \cite{osting2017spectral, spielman2008graph}, where it is shown to be the highest density for the sparse case. For each pair of parameters \( \left(m_0, \faktor{ m_2 }{ q(m_1) } \right) \), \( N = 25\) simplicial complexes are generated; triangle weight profile \( w_2(t) \) is given by the following two scenarios:
\begin{enumerate}[label=(\arabic*), leftmargin=*]
      \item indendent triangle weights with bi-modal imbalance, where \( w_2(t) \sim \mc N(1, 1/3) \) for the original triangulation \(t \in \mc K_\Delta \) and \( w_2(t) \sim \mc N(1/2, 1/6) \) otherwise, \Cref{fig:bimodal};
      \item dependent triangle weights through the min-rule: $w_2(t) = \min \left\{\right. w_1(e_1), \allowbreak w_2(e_2), \allowbreak w_3(e_3) \left.\right\} $  for \( t = (e_1, e_2, e_3) \) and edge weights are folded normal variables, \(w_1(e_i) \sim |\mc N| [0, 1]\),  \Cref{fig:minrule}. 
\end{enumerate}
For each weight profile we measure the condition number (\(\kappa_+(C^+ P_\Sigma \Lu 1 P_\Sigma^\top (C^+)^{\top} ) \) vs \( \kappa_+(\Lu 1)\)), \Cref{fig:bimodal,fig:minrule}, left, and the corresponding number of CGLS iteration, \Cref{fig:bimodal,fig:minrule}, right. In the case of the min-rule, we provide a high-performance test for matrices up to \( 10^5 \) in size.

In the case of the independent triangle weights, \Cref{fig:bimodal}, \algname~preconditioning shows gains in \( \kappa_+ \) for the first, sparser, part of the simplicial complexes; conversely, for the min-rule profile induced by the folded normal edges' weights, \Cref{fig:minrule}, developed method outperforms the original system for all tested sparsity patterns \( \faktor{ m_2 }{ q(m_1) } \). Noticeably, \algname~preconditioning performs better in terms of the actual CGLS  iterations, \Cref{fig:bimodal,fig:scheme}, right, than in terms of \(\kappa_+\), and, hence, significantly speeds up the iterative solver for \( \Lu 1 \).

\begin{figure}[hbtp]
      \centering
      \includegraphics[width=1.0\columnwidth]{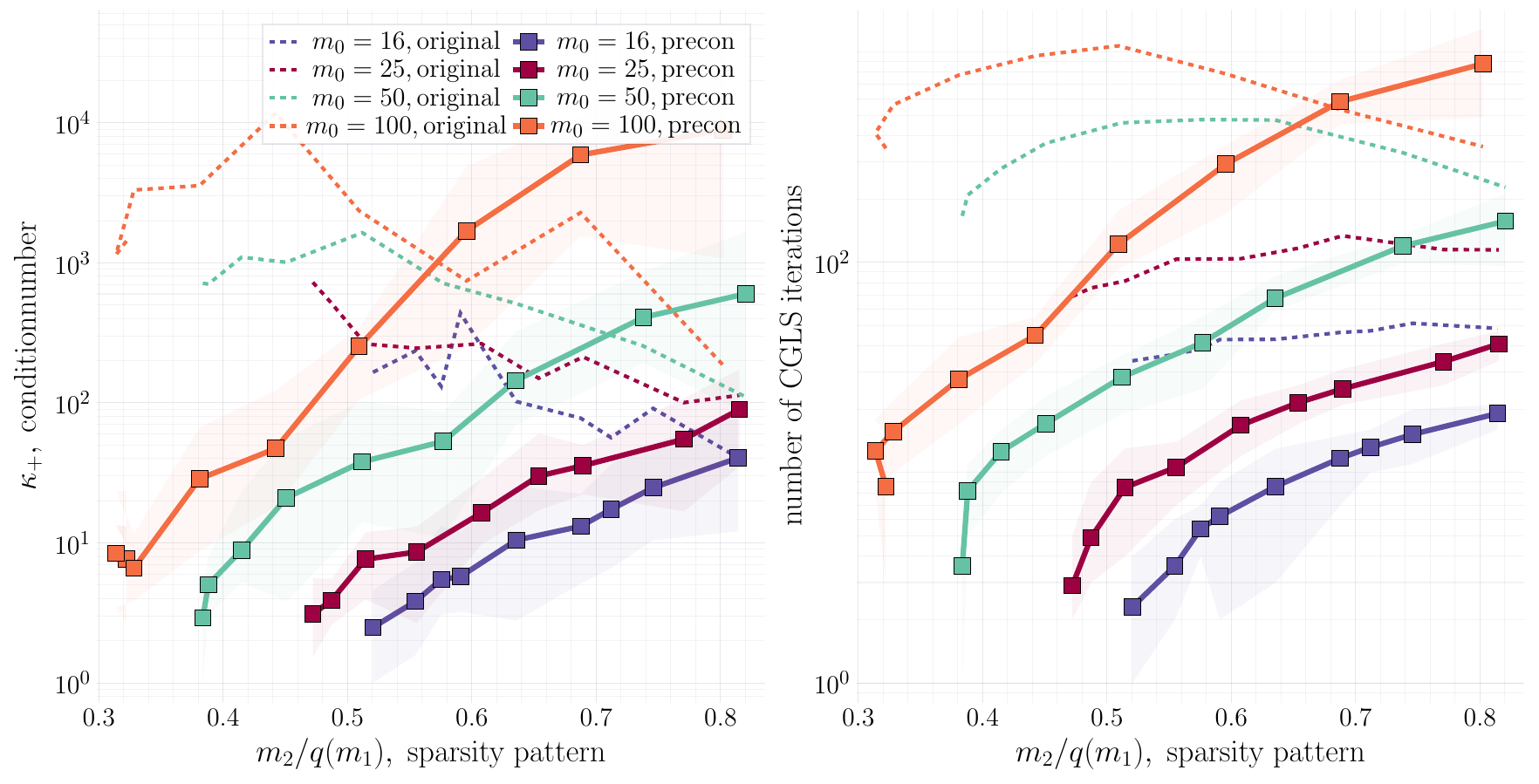}
      \caption{\small Preconditioning quality for enriched triangulations with a varying number of vertices \( m_0 = 16, 25, 50, 100\) and sparsity patterns \( \faktor{ m_2 }{ q(m_1) } \) and independent bi-modal weight profile: condition numbers \( \kappa_+ \) on the left and the number of CGLS iterations on the right. Average results among \( 25 \) generations are shown in solid (\algname) and in dash (original system); colored areas around the solid line show the dispersion among the generated complexes.
            \label{fig:bimodal}
      }
\end{figure}

\begin{figure}[hbtp]
      \centering
      \includegraphics[width=1.0\columnwidth]{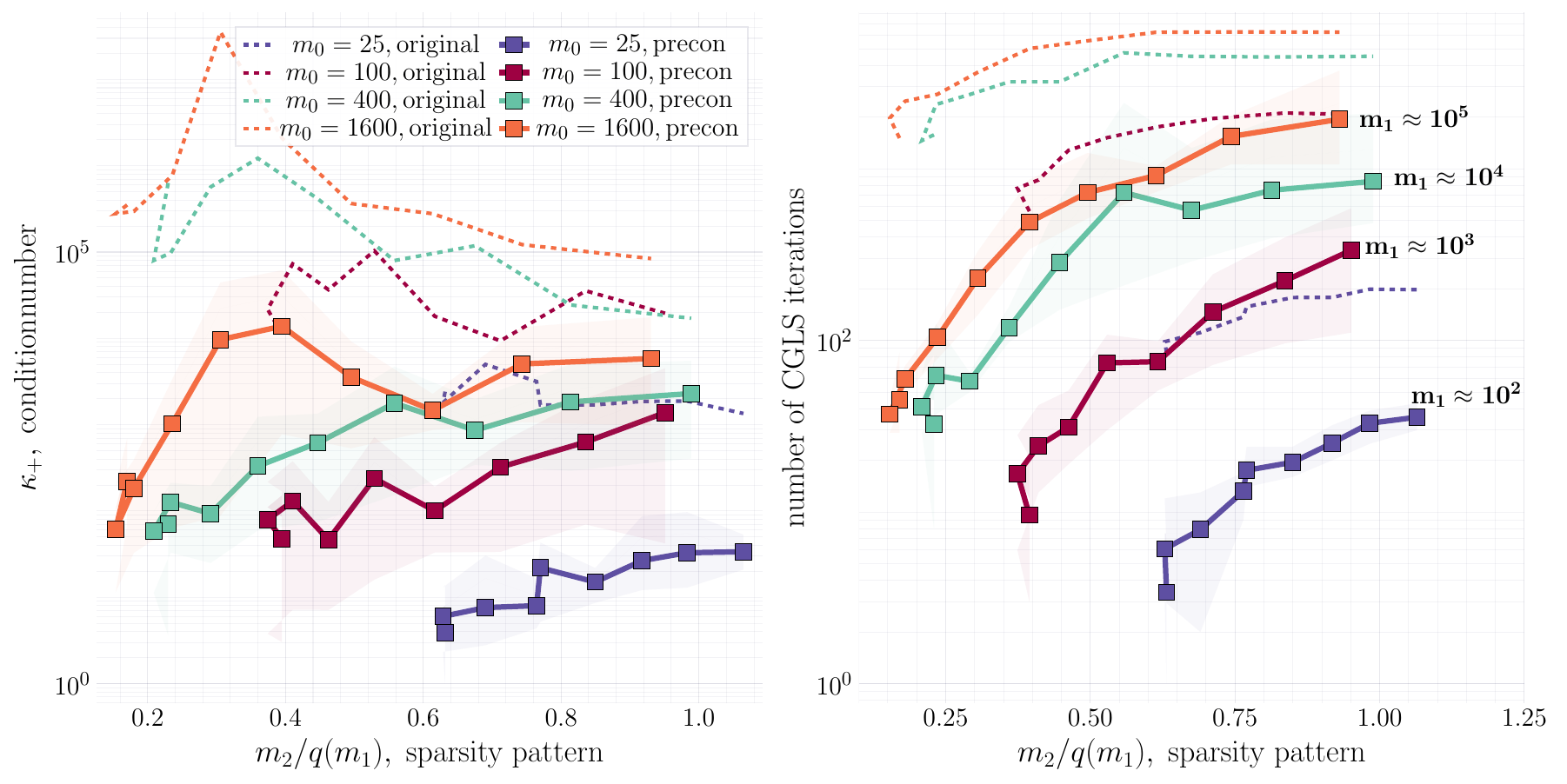}
      \caption{\small Preconditioning quality for enriched triangulations with a varying number of vertices \( m_0 = 25, 100, 400, 1600\) and sparsity patterns \( \faktor{ m_2 }{ q(m_1) } \) and dependent min-rule weight profile with folded normal edge weights: condition numbers \( \kappa_+ \) on the left and the number of CGLS iterations on the right. Average results among \( 25 \) generations are shown in solid (\algname) and in dash (original system); colored areas around the solid line show the dispersion among the generated complexes.           \label{fig:minrule}
      }
\end{figure}

Finally, we demonstrate the comparative performance with the shifted incomplete Cholesky preconditioner, \( C_\alpha \), \Cref{fig:ichol}; here we are forced to guarantee trivial \(0\)- and \(1\)-homologies, so no edges are eliminated in the triangulation, \( d = 0 \), and we check the kernel of \(L_1\) for triviality after the generation of \( \mc K \). Similarly to the previous results, preconditioning with the shifted \texttt{ichol} \( C_\alpha \) is more efficient than \algname~preconditioning for the ``densest part'' of the considered simplicial complexes which means that our developed method still performs better in case of the sparser \(\mc K \). Moreover, the applicability of the shifted \texttt{ichol} is limited to the cases of trivial homologies which is not the case for \algname~preconditioning.

\begin{figure}[hbtp]
      \centering
      \includegraphics[width=1.0\columnwidth]{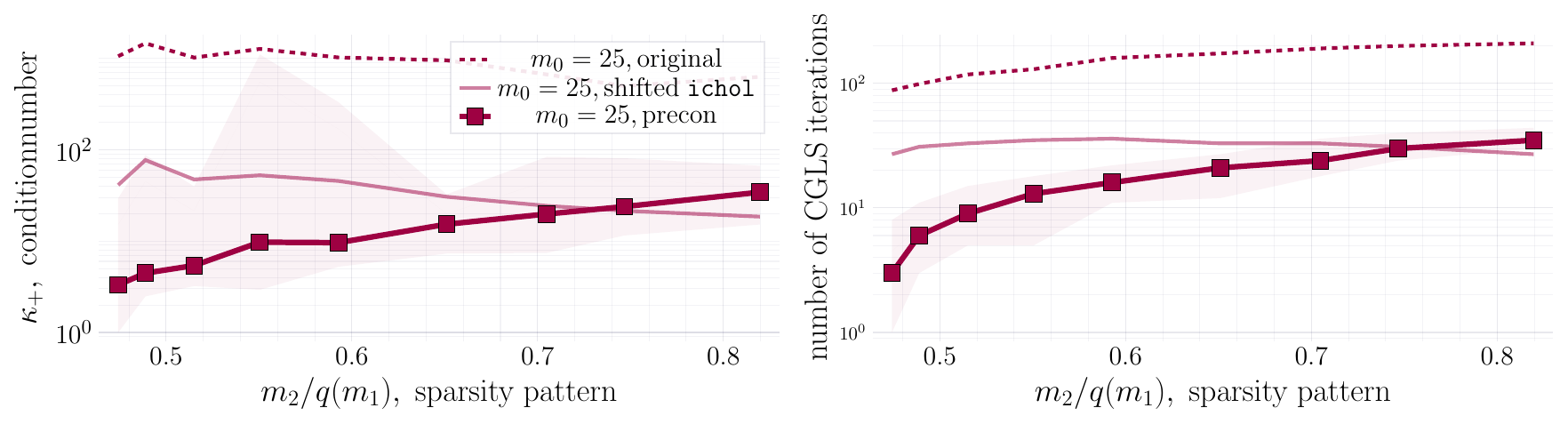}
      \caption{ \small  Comparison of the preconditioning quality between \algname(solid), shifted \texttt{ichol} (semi-transparent) and original system (dashed) for the enriched triangulation on \( m_0 = 25\) vertices and varying sparsity patterns \( \nu \) and dependent min-rule weight profile with uniform edge weights: condition numbers \( \kappa_+ \) on the left and the number of CGLS iterations on the right. Average results among \( 25 \) generations are shown in solid (\algname{} and \texttt{ichol}) and in dash (original system); colored areas around the solid line show the dispersion among the generated complexes.           \label{fig:ichol}
      }
\end{figure}

\subsection*{Acknowledgments}

N.G.\ acknowledges that his research was supported by funds from the Italian 
MUR (Ministero dell'Universit\`a e della Ricerca) within the PRIN 2022 Project ``Advanced numerical methods for time dependent parametric partial 
differential equations with applications'' and the Pro3 joint project entitled
``Calcolo scientifico per le scienze naturali, sociali e applicazioni: sviluppo metodologico e tecnologico''.  
N.G. and F.T. acknowledge support from MUR-PRO3 grant STANDS and PRIN-PNRR grant FIN4GEO. F.T. also acknowledges support from INdAM – GNCS Project, number CUP\_E53C22001930001. The authors are members of INdAM - GNCS which they thank for support.


\end{document}

%% file: ex_shared.tex
\usepackage{lipsum}
\usepackage{amsfonts}
\usepackage{graphicx}
\usepackage{algorithm}
\usepackage{algpseudocode}
\ifpdf
  \DeclareGraphicsExtensions{.eps,.pdf,.png,.jpg}
\else
  \DeclareGraphicsExtensions{.eps}
\fi

\newsiamremark{remark}{Remark}
\newsiamremark{example}{Example}
\newsiamremark{hypothesis}{Hypothesis}
\crefname{hypothesis}{Hypothesis}{Hypotheses}
\newsiamthm{claim}{Claim}
\newsiamthm{problem}{Problem}
\crefname{problem}{Problem}{Problem}

\headers{Cholesky-like preconditioner for Hodge Laplacian
}{A. Savostianov, F. Tudisco, N. Guglielmi}

\title{Cholesky-like Preconditioner for Hodge Laplacians via Heavy Collapsible Subcomplex
\thanks{Submitted to the editors on DATE.}
}

\author{Anton Savostianov\thanks{Gran Sasso Science Institute, L'Aquila, Italy 
(\email{anton.savostianov@gssi.it})}
\and Francesco Tudisco\thanks{
Maxwell Institute \& School of Mathematics, University of Edinburgh, UK (\email{f.tudisco@ed.ac.uk})
}
\and Nicola Guglielmi\thanks{Gran Sasso Science Institute, L'Aquila, Italy 
(\email{nicola.guglielmi@gssi.it})}
}

\usepackage{amsopn}
\DeclareMathOperator{\diag}{diag} 
\DeclareMathOperator{\im}{im}
\DeclareMathOperator{\argmin}{arg\,min}

\makeatletter
\newcommand*{\addFileDependency}[1]{
  \typeout{(#1)}
  \@addtofilelist{#1}
  \IfFileExists{#1}{}{\typeout{No file #1.}}
}
\makeatother

\usepackage{amsbsy,amssymb}
\usepackage{faktor}

\newcommand*{\mc}[1]{ \mathcal{#1} }
\renewcommand*{\b}[1]{ \pmb{#1} }
\usepackage{dsfont}
\newcommand*{\ds}[1]{ \mathds{#1} }
\newcommand*{\V}[1]{ \mc V_{#1}( \mc K) }

\newcommand{\vn}{\varnothing}

\definecolor{lavender}{HTML}{7287fd}

\usepackage{xfrac}
\usepackage{enumitem}

\newcommand*{\Lu}[1]{L_{#1}^{\uparrow}}
\newcommand*{\Ld}[1]{L_{#1}^{\downarrow}}

\newcommand*{\wh}[1]{\widehat{#1}}

\usepackage{tikz}
\usepackage{pgfplots}
\usepackage{tikz-network}
\pgfplotsset{compat=1.5}
\usetikzlibrary{intersections}

\definecolor{bananamania}{rgb}{0.98, 0.91, 0.71}
\definecolor{burntsienna}{rgb}{0.91, 0.45, 0.32}
\definecolor{airforceblue}{rgb}{0.36, 0.54, 0.66}
\definecolor{liberty}{HTML}{5158BB}
\definecolor{junglegreen}{rgb}{0.16, 0.67, 0.53}
\definecolor{persimmon}{HTML}{DE5A02}

\usepackage{comment}

\usetikzlibrary{arrows.meta}
\usetikzlibrary{backgrounds}
\usepgfplotslibrary{patchplots}
\usepgfplotslibrary{fillbetween}
\pgfplotsset{%
    layers/standard/.define layer set={%
        background,axis background,axis grid,axis ticks,axis lines,axis tick labels,pre main,main,axis descriptions,axis foreground%
    }{
        grid style={/pgfplots/on layer=axis grid},%
        tick style={/pgfplots/on layer=axis ticks},%
        axis line style={/pgfplots/on layer=axis lines},%
        label style={/pgfplots/on layer=axis descriptions},%
        legend style={/pgfplots/on layer=axis descriptions},%
        title style={/pgfplots/on layer=axis descriptions},%
        colorbar style={/pgfplots/on layer=axis descriptions},%
        ticklabel style={/pgfplots/on layer=axis tick labels},%
        axis background@ style={/pgfplots/on layer=axis background},%
        3d box foreground style={/pgfplots/on layer=axis foreground},%
    },
}

\newenvironment{mt}{\begin{pmatrix}}{\end{pmatrix}}

\usetikzlibrary{arrows.meta}
\usetikzlibrary{backgrounds}
\usepgfplotslibrary{patchplots}
\usepgfplotslibrary{fillbetween}
\pgfplotsset{%
     layers/standard/.define layer set={%
         background,axis background,axis grid,axis ticks,axis lines,axis tick labels,pre main,main,axis descriptions,axis foreground%
     }{
         grid style={/pgfplots/on layer=axis grid},%
         tick style={/pgfplots/on layer=axis ticks},%
         axis line style={/pgfplots/on layer=axis lines},%
         label style={/pgfplots/on layer=axis descriptions},%
         legend style={/pgfplots/on layer=axis descriptions},%
         title style={/pgfplots/on layer=axis descriptions},%
         colorbar style={/pgfplots/on layer=axis descriptions},%
                  ticklabel style={/pgfplots/on layer=axis tick labels},%
                  axis background@ style={/pgfplots/on layer=axis background},%
                  3d box foreground style={/pgfplots/on layer=axis foreground},%
              },
          }

          \usepackage{caption}
          \usepackage{subcaption}

\newcommand{\algname}{\texttt{HeCS}}


%% file: arxiv.bbl
\begin{thebibliography}{10}

\bibitem{benson2016higher}
{\sc A.~R. Benson, D.~F. Gleich, and J.~Leskovec}, {\em Higher-order
  organization of complex networks}, Science, 353 (2016), pp.~163--166.

\bibitem{bjorck1998stability}
{\sc {\AA}.~Bj{\"o}rck, T.~Elfving, and Z.~Strakos}, {\em Stability of
  conjugate gradient and lanczos methods for linear least squares problems},
  SIAM Journal on Matrix Analysis and Applications, 19 (1998), pp.~720--736.

\bibitem{black2022computational}
{\sc M.~Black, W.~Maxwell, A.~Nayyeri, and E.~Winkelman}, {\em Computational
  topology in a collapsing universe: {L}aplacians, homology, cohomology}, in
  Proceedings of the 2022 Annual ACM-SIAM Symposium on Discrete Algorithms
  (SODA), SIAM, 2022, pp.~226--251.

\bibitem{chen2021decomposition}
{\sc Y.-C. Chen and M.~Meila}, {\em The decomposition of the higher-order
  homology embedding constructed from the $ k $-laplacian}, Advances in Neural
  Information Processing Systems, 34 (2021), pp.~15695--15709.

\bibitem{chen2021helmholtzian}
{\sc Y.-C. Chen, M.~Meila, and I.~G. Kevrekidis}, {\em Helmholtzian eigenmap:
  Topological feature discovery \& edge flow learning from point cloud data},
  CoRR, abs/2103.07626 (2021).

\bibitem{cohen2014solving}
{\sc M.~B. Cohen, B.~T. Fasy, G.~L. Miller, A.~Nayyeri, R.~Peng, and
  N.~Walkington}, {\em Solving 1-{L}aplacians in nearly linear time:
  {C}ollapsing and expanding a topological ball}, in Proceedings of the
  Twenty-Fifth Annual ACM-SIAM Symposium on Discrete Algorithms, SIAM, 2014,
  pp.~204--216.

\bibitem{demmel1997}
{\sc J.~W. Demmel}, {\em Iterative Methods for Eigenvalue Problems},
  ch.~Applied Numerical Linear Algebra, pp.~361--387.

\bibitem{ebli2019notion}
{\sc S.~Ebli and G.~Spreemann}, {\em A notion of harmonic clustering in
  simplicial complexes}, in 2019 18th IEEE International Conference On Machine
  Learning And Applications (ICMLA), IEEE, 2019, pp.~1083--1090.

\bibitem{gambuzza2021stability}
{\sc L.~V. Gambuzza, F.~Di~Patti, L.~Gallo, S.~Lepri, M.~Romance, R.~Criado,
  M.~Frasca, V.~Latora, and S.~Boccaletti}, {\em Stability of synchronization
  in simplicial complexes}, Nature communications, 12 (2021), p.~1255.

\bibitem{golub2013matrix}
{\sc G.~H. Golub and C.~F. Van~Loan}, {\em Matrix computations}, JHU press,
  2013.

\bibitem{grande2023disentangling}
{\sc V.~P. Grande and M.~T. Schaub}, {\em Disentangling the spectral properties
  of the {H}odge {L}aplacian: Not all small eigenvalues are equal},
  arXiv:2311.14427,  (2023).

\bibitem{grande2023topological}
{\sc V.~P. Grande and M.~T. Schaub}, {\em Topological point cloud clustering},
  arXiv:2303.16716,  (2023).

\bibitem{guglielmi2023quantifying}
{\sc N.~Guglielmi, A.~Savostianov, and F.~Tudisco}, {\em Quantifying the
  {{Structural Stability}} of {{Simplicial Homology}}}, Journal of Scientific
  Computing, 97 (2023), p.~2.

\bibitem{hestenes1952methods}
{\sc M.~R. Hestenes, E.~Stiefel, et~al.}, {\em Methods of conjugate gradients
  for solving linear systems}, Journal of research of the National Bureau of
  Standards, 49 (1952), pp.~409--436.

\bibitem{higham1990analysis}
{\sc N.~J. Higham}, {\em Reliable Numerical Computation}, Oxford University
  Press, 1990, ch.~Analysis of the {C}holesky decomposition of a semi-definite
  matrix.

\bibitem{Kyng2016}
{\sc R.~Kyng and S.~Sachdeva}, {\em Approximate gaussian elimination for
  {L}aplacians: fast, sparse, and simple}, in 2016 IEEE 57th Annual Symposium
  on Foundations of Computer Science (FOCS), IEEE, 2016, pp.~573--582.

\bibitem{lee2019coidentification}
{\sc H.~Lee, M.~K. Chung, H.~Kang, H.~Choi, S.~Ha, Y.~Huh, E.~Kim, and D.~S.
  Lee}, {\em Coidentification of group-level hole structures in brain networks
  via {H}odge {L}aplacian}, in Medical Image Computing and Computer Assisted
  Intervention--MICCAI 2019: 22nd International Conference, Shenzhen, China,
  October 13--17, 2019, Proceedings, Part IV 22, Springer, 2019, pp.~674--682.

\bibitem{Lim15}
{\sc L.-H. Lim}, {\em Hodge laplacians on graphs}, Siam Review, 62 (2020),
  pp.~685--715.

\bibitem{lofano2019worst}
{\sc D.~Lofano and A.~Newman}, {\em The worst way to collapse a simplex},
  Israel Journal of Mathematics, 244 (2021), pp.~625--647.

\bibitem{manteuffel1980incomplete}
{\sc T.~A. Manteuffel}, {\em An incomplete factorization technique for positive
  definite linear systems}, Mathematics of computation, 34 (1980),
  pp.~473--497.

\bibitem{osting2017spectral}
{\sc B.~Osting, S.~Palande, and B.~Wang}, {\em Spectral sparsification of
  simplicial complexes for clustering and label propagation}, Journal of
  computational geometry, 11 (2022).

\bibitem{ribandogros2023combinatorial}
{\sc E.~Ribando-Gros, R.~Wang, J.~Chen, Y.~Tong, and G.-W. Wei}, {\em
  Combinatorial and {H}odge {L}aplacians: Similarity and difference}, 2023.

\bibitem{schaub2019random}
{\sc M.~T. Schaub, A.~R. Benson, P.~Horn, G.~Lippner, and A.~Jadbabaie}, {\em
  Random walks on simplicial complexes and the normalized {H}odge
  1-{L}aplacian}, SIAM Review, 62 (2020), pp.~353--391.

\bibitem{schaub2022signal}
{\sc M.~T. Schaub, J.-B. Seby, F.~Frantzen, T.~M. Roddenberry, Y.~Zhu, and
  S.~Segarra}, {\em Signal processing on simplicial complexes}, in Higher-Order
  Systems, Springer, 2022, pp.~301--328.

\bibitem{spielman2008graph}
{\sc D.~A. Spielman and N.~Srivastava}, {\em Graph sparsification by effective
  resistances}, vol.~40, SIAM, 2011, pp.~1913--1926.

\bibitem{tancer2008dcollapse}
{\sc M.~Tancer}, {\em d-collapsibility is np-complete for d greater or equal to
  4}, Chicago Journal OF Theoretical Computer Science, 3 (2010), pp.~1--28.

\bibitem{tancer2016recognition}
{\sc M.~Tancer}, {\em Recognition of collapsible complexes is {N}{P}-complete},
  Discrete \& Computational Geometry, 55 (2016), pp.~21--38.

\bibitem{torres2020simplicial}
{\sc J.~J. Torres and G.~Bianconi}, {\em Simplicial complexes: higher-order
  spectral dimension and dynamics}, Journal of Physics: Complexity, 1 (2020),
  p.~015002.

\bibitem{whitehead1939simplicial}
{\sc J.~H.~C. Whitehead}, {\em Simplicial spaces, nuclei and m-groups},
  Proceedings of the London mathematical society, 2 (1939), pp.~243--327.

\end{thebibliography}
